\documentclass[aap,MSNbibl,seceqn,dvips]{arximspdf}
\usepackage{mathbh}
\usepackage{mathrsfs}

\doi{10.1214/13-AAP990} 
\volume{25}
\issue{1}
\pubyear{2015}
\firstpage{116}
\lastpage{149}

\makeatletter
\newproclaim{definition}{Definition}[section]
\newproclaim{remark}[definition]{Remark}
\newproclaim{example}[definition]{Example}
\newtheorem{theorem}[definition]{Theorem}
\newtheorem{lemma}[definition]{Lemma}
\newtheorem{corollary}[definition]{Corollary}
\newtheorem{proposition}[definition]{Proposition}
\newproclaim{assumption}[definition]{Assumption}
\makeatother

\begin{document}
\begin{frontmatter}

\title{Regularity conditions in the realisability problem with
applications to point processes and\\ random closed sets}
\runtitle{Realisability problem}

\begin{aug}
\author[A]{\fnms{Raphael} \snm{Lachieze-Rey}\ead[label=e1]{lr.raphael@gmail.com}}
\and
\author[B]{\fnms{Ilya} \snm{Molchanov}\corref{}\ead[label=e2]{ilya@stat.unibe.ch}\thanksref{T2}}
\runauthor{R. Lachieze-Rey and I. Molchanov}
\affiliation{Universit\'e Paris Descartes and University of Bern}
\address[A]{Laboratoire MAP 5\\
Universit\'e Paris Descartes\\
45 Rue des saints-p\`eres\\
75006 Paris\\
France\\
\printead{e1}} 
\address[B]{Institute of Mathematical Statistics\\
\quad and Actuarial Science\\
University of Bern\\
Sidlerstrasse 5\\
3012 Bern\\
Switzerland\\
\printead{e2}}
\end{aug}
\thankstext{T2}{Supported by Swiss National Foundation
Grant No.~200021-126503. The revision was done while
the second author held the Chair of Excellence at the University
Carlos III of Madrid supported by the Santander Bank.}

\received{\smonth{3} \syear{2011}}
\revised{\smonth{9} \syear{2013}}

%
\begin{abstract}
We study existence of random elements with partially specified
distributions. The technique relies on the existence of a positive
extension for linear functionals accompanied by additional
conditions that ensure the regularity of the extension needed for
interpreting it as a probability measure. It is shown in which case
the extension can be chosen to possess some invariance properties.

The results are applied to the existence of point processes with
given correlation measure and random closed sets with given
two-point covering \mbox{function} or contact distribution function. It is
shown that the regularity condition can be efficiently checked in
many cases in order to ensure that the obtained point processes are
indeed locally finite and random sets have closed realisations.
\end{abstract}

%
\begin{keyword}[class=AMS]
\kwd[Primary ]{60D05}
\kwd[; secondary ]{28C05}
\kwd{46A40}
\kwd{47B65}
\kwd{60G55}
\kwd{74A40}
\kwd{82D30}
\end{keyword}
\begin{keyword}
\kwd{Point process}
\kwd{correlation measure}
\kwd{random closed set}
\kwd{two-point covering probability}
\kwd{contact distribution function}
\kwd{realisability}
\end{keyword}

\end{frontmatter}

\section{Introduction}\label{secintroduction}

Defining the distribution of a random element $\xi$ in a topological
space $\mathcal{X}$ is equivalent to specialising the expected values
for all
bounded continuous functions $\mathsf{g}(\xi)$. These expected values
define a linear functional $\Phi(\mathsf{g})=\mathbf{E}\mathsf
{g}(\xi)$ on the space of
bounded continuous functions $\mathsf{g}\dvtx \mathcal{X}\mapsto\mathbb
{R}$. It is well known
that a functional $\Phi$ indeed corresponds to a random element if and
only if $\Phi$ is positive [i.e., $\Phi(\mathsf{g})\geq0$ if
$\mathsf{g}$ is
nonnegative] and upper semi-continuous [i.e., $\Phi(\mathsf
{g}_n)\downarrow0$
if $\mathsf{g}_n\downarrow0$]; see, for example, \cite{whit92}.

Below we consider the case of functional $\Phi$ defined only on some
functions on $\mathcal{X}$ and address the \emph{realisability} of
$\Phi$,
that is, the mere existence of a random element $\xi$ such that
$\Phi(\mathsf{g})=\mathbf{E}\mathsf{g}(\xi)$ for $\mathsf{g}$
from the chosen family $\mathsf{G}$ of
functions. The uniqueness is not on the agenda, since typically the
family $\mathsf{G}$ will not suffice to uniquely specify the distribution
of $\xi$. A classical example of this setting is the existence of a
probability distribution with given marginals; see \cite{kel64}. The
present paper focuses on some geometric instances of the problem. We
will see that in most cases the answer to the existence problem
consists of the two main steps.
\begin{longlist}[2.]
\item[1.] (\textit{Positivity}) Checking the positivity condition on $\Phi$---in
most cases this requires checking a system of inequalities,
which is a serious (but unavoidable) computational burden.

\item[2.] (\textit{Regularity}) Ensuring that the extended functional is
regular (namely, upper semi-continuous) and so defines a
$\sigma$-additive measure.
\end{longlist}
The first step ensures that it is possible to extend functional $\Phi$
positively from a certain family of functions to a wider family.
In this work, we put the emphasis on the latter step---checking the
regularity condition, leaving aside the computational difficulties
arising from validating the positivity assumption.

The use of positive extension techniques (that goes back to
Kantorovitch) in the framework of stochastic geometry was
pioneered by Kuna, Lebowitz and Speer
\cite{kunlebspeer11} in application to point processes, which
greatly inspired the current work. In this paper, we establish the
general nature of an idea proposed in \cite{kunlebspeer11} and show
how it leads to various further realisability results. The new idea
is to introduce an additional function, what we call the regularity
modulus, and to formulate sufficient and necessary conditions in terms
of a positive extension of a functional onto the linear space containing
the regularity modulus and requiring only a priori integrability of
the regularity modulus.

We concentrate on two basic examples of the realisability problem: the
existence of point processes with given correlation (factorial moment)
measure and the existence of a random closed set with given two-point
coverage probabilities or contact distribution functions. The
introduction to the realisability issue for point processes is
available in several papers by Kuna, Lebowitz and Speer
\cite{kunlebspeer07,kunlebspeer11}; see also Section~\ref{secpoint-processes} of this paper. The realisability problem for
random closed sets has been widely studied in physics and material
science literature; see
\cite{jiaostiltor07,mar98,tor99,tor06,torstel82} and in particular
the comprehensive monograph by Torquato \cite{tor02b} and a recent
survey by Quintanilla \cite{quin08}. If $\xi$ is a random closed
set (see Section~\ref{secreals-under-smoothn} for formal definitions) in
a locally compact metric space $\mathbb{X}$, its
\emph{one-point covering functions} is defined by
\[
p_x=\mathbf{P}\{x\in\xi\},\qquad x\in\mathbb{X}.
\]
It is easy to characterise all one-point covering functions of random
closed sets as follows.

%
\begin{theorem}
\label{throp}
A function $p_x$, $x\in\mathbb{X}$, with values in $[0,1]$ is the
one-point covering function of a random closed set if and only if
$p$ is upper semi-continuous.
\end{theorem}

The upper semi-continuity of the one-point covering function of a
random closed set $\xi$ is a straightforward consequence of the upper
semi-continuity property of the capacity functional of a random closed
set; see \cite{mo1}, Section~1.1.2. Conversely, the function $p$ from
the theorem is realised (e.g.) as the one-point covering
function of the random set $\xi=\{x\dvtx  p_{x}\geq v\}$ where $v$ is a
uniformly distributed variable (the details are left to the reader).

It is considerably more complicated to characterise \emph{two-point
covering functions}
\[
p_{x,y}=\mathbf{P}\{x,y\in\xi\},\qquad x,y\in\mathbb{X}.
\]
In view of applications to modelling of random media, it is often
assumed that $\xi$ is a stationary set in $\mathbb{R}^d$, so that the
one-point
covering function is constant and the two-point covering function
$p_{x,y}$ depends only on $x-y$.
Since a random closed set can be considered as an upper semi-continuous
indicator function, the realisability problem for the two-point
covering function can be rephrased as follows:

\begin{quote}
Characterise covariance functions of (stationary) upper
semi-continuous random functions with values in $\{0,1\}$.
\end{quote}
These covariances are obviously a sub-family of positive semi-definite
functions. Without the upper semi-continuity requirement, this problem,
of combinatorial nature, was solved by McMillan \cite{mcm55} and
Shepp \cite{shep63,shep67} using the extension \mbox{argument} from
\cite{kel64}. More exactly, they normalised indicators by letting
them take values $+$1 or $-$1 and assumed that the mean is zero. Their
result does not rely on the topological structure of the underlying
space and so does not necessarily lead to an upper semi-continuous
indicator function.

%
\begin{example}
\label{exproduct}
Let $p_{x,y}=\frac{1}{4}$ and let $p_{x}=\frac{1}{2}$ for all
$x,y\in\mathbb{R}$. While this two-point covering function corresponds,
for example, to the indicator field with independent values, it cannot be
obtained as the two-point covering function of a random closed set;
see Proposition~\ref{propproduct-form}.
\end{example}

Even leaving aside the upper semi-continuity property, the
McMillan--Shepp condition involves a family of corner-positive
matrices, which is poorly understood. As a result, its practical use
to check the realisability for random media is rather limited. A
number of authors have attempted to come up with simpler (but only
necessary) conditions; see, for example,
\cite{jiaostiltor07,mater93,quin08,tor06}. Another set of
conditions for joint distributions of binary random variables is
formulated in \cite{sharibr02} in terms of the corresponding copulas.

The realisability problem can be also posed for point processes in
terms of their moment measures. In case of moment measures of
arbitrary order, it has been solved by Lenard \cite{len75a,len75b}.
The case of moment measures up to the second order has been studied by
Kuna, Lebowitz and Speer \cite{kunlebspeer07}, whose
recent paper \cite{kunlebspeer11} contains (among other results) a
\emph{complete} solution of this realisability problem for point
processes with finite third-order moments and hard-core type
conditions with fixed exclusion distance. The results of
\cite{kunlebspeer11} can be extended to higher order moment
measures, as was explicitly indicated there. Again, the positivity
condition of \cite{kunlebspeer11} is extremely difficult to verify,
even more complicated than the original condition for point processes
because of new polynomial functionals involved in the positivity
condition.

The paper is organised as follows.
Section~\ref{secextend-posit-funct} presents a series of general
results on regular extensions and also invariant extensions (relevant
for the existence of stationary random elements). These results form
the theoretical backbone of our study, and are new even in the
abstract setting of extending general positive linear functionals.

Section~\ref{secpoint-processes} presents a number of realisability
conditions for correlation measures of point processes that
considerably extend the results of \cite{kunlebspeer11} by relaxing
the moment and hardcore conditions. One of our most important results
is Theorem~\ref{thrreal-pp-psi} that shows how to split the
positivity and regularity conditions, so that the latter can be
efficiently checked. The importance of the packing number in
relation to realisability conditions for hard-core point processes is
also explained.

Section~\ref{secreals-under-smoothn} deals with the realisability
problem for two-point covering probabilities of random sets. The
closedness of the corresponding random set can be ensured by imposing
appropriate regularity conditions. Section~\ref{seccont-distr-funct}
addresses a further variant of the realisability problem that
involves contact distribution functions of random sets.

The notational convention is that the carrier space is denoted as
$\mathbb{X}$ (e.g., $\mathbb{R},\mathbb{R}^d$), points in the
carrier space are $x,y$,
subsets of carrier spaces are denoted by capitals $X,Y,F$ (while
$Y$ is reserved for counting measures identified with corresponding
support sets), the families of sets (or families of counting measures)
as $\mathcal{X},\mathcal{N},\mathcal{F}$ (while in Section~\ref{secextend-posit-funct} $\mathcal{X}$
denotes also a rather general space) and random element in these
spaces (random sets or point processes) as $\xi$,
real functions acting on $\mathcal{X},\mathcal{N},\mathcal{F}$ are
$\mathsf{g},\mathsf{v}$ and families of
such functions are $\mathsf{G},\mathsf{E},\mathsf{V}$, a functional on
$\mathsf{G},\mathsf{E},\mathsf{V}$ is denoted by $\Phi$, real
numbers are
denoted by $t,r,\lambda$, while $c$ denotes a generic constant and at
the same time the corresponding constant function.

\section{Extending positive functionals}\label{secextend-posit-funct}

Fundamental results about the extension of positive operators form the
heart of our main results, and are necessary to understand the
machinery of the proofs. Nevertheless, the results of the subsequent
sections can be understood without
Section~\ref{secextend-posit-funct}, with the exception of
Definition~\ref{defreg-mod}.

\subsection{General extension theorems}\label{secgener-extens-theor}

Consider a \emph{vector lattice} $\mathsf{E}$, that is a linear space
with a
partial order and such that for any $\mathsf{v}_1,\mathsf{v}_2\in
\mathsf{E}$ their
maximum $\mathsf{v}_1\vee\mathsf{v}_2$ also belongs to $\mathsf
{E}$. The absolute
value $|\mathsf{v}|$ of $\mathsf{v}$ is defined as the sum of
$\mathsf{v}\vee0$ and
$(-\mathsf{v})\vee0$.

Let $\mathsf{G}$ be a \emph{vector subspace} of $\mathsf{E}$, which
is not
necessarily a lattice itself, that is $\mathsf{G}$ may be not closed with
respect to the maximum operation. We say that $\mathsf{G}$
\emph{majorises} $\mathsf{E}$ if each $\mathsf{v}\in\mathsf{E}$ satisfies
$|\mathsf{v}|\leq\mathsf{g}$ for some $\mathsf{g}\in\mathsf{G}$.
A real-valued functional
$\Phi$ defined on $\mathsf{E}$ (resp., $\mathsf{G}$) is said to be
\emph{positive} if $\Phi(\mathsf{v})\geq0$ whenever $\mathsf
{v}\geq0$ and
$\mathsf{v}\in\mathsf{E}$ (resp., $\mathsf{v}\in\mathsf{G}$). A
functional defined on
$\mathsf{E}$ is said to be an \emph{extension} of \mbox{$\Phi\dvtx \mathsf
{G}\mapsto\mathbb{R}$}
if it coincides with $\Phi$ on $\mathsf{G}$. The extended $\Phi$ is
always denoted by the same letter. The following result about
extension of positive functionals goes back to Kantorovich.

%
\begin{theorem}[(See \cite{alipbor06}, Theorem~8.12 and \cite{vul67},
Theorem~X.3.1)]\label{thrext}
Assume that $\mathsf{G}$ is a majorising vector subspace of a vector
lattice $\mathsf{E}$. Then each positive linear functional on $\mathsf{G}$
admits a positive extension on the whole $\mathsf{E}$.
\end{theorem}

If $\mathsf{G}$ is a lattice itself, then it is possible to gain much
more control over the extension of $\Phi$, for example, a continuous
functional admits a continuous extension, see \cite{vul67}, Section~X.5.
On the contrary, very little is known about regularity properties of
the extension if $\mathsf{G}$ is not a lattice.

In the following, we assume that $\mathsf{G}$ and $\mathsf{E}$ are families
of functions $g$ on a certain space $\mathcal{X}$. If $\mathsf{G}$ contains
constant functions, the positivity of $\Phi$ over $\mathsf{G}$ can be
equivalently formulated as
%
\begin{equation}
\label{eqphi-positivity} \Phi(\mathsf{g})\geq\inf_{X\in\mathcal
{X}}\mathsf{g}(X).
\end{equation}
This equivalence is a particular case of the following result for
$\chi=0$ [replace $\mathsf{g}$ with $-\mathsf{g}$ in
(\ref{eqpositivity-with-constants})].

%
\begin{proposition}
\label{corsingle-ext}
Assume that vector space $\mathsf{G}$ contains constant functions and
denote by $\mathsf{G}\setminus\mathbb{R}$ the family of nonconstant
functions
from $\mathsf{G}$. If $\chi$ is any nonnegative function on
$\mathcal{X}$, then a linear functional $\Phi$ on $\mathsf{G}$
admits a
positive extension on $\mathsf{G}+\mathbb{R}\chi$ with $\Phi(\chi
)=r$ if and
only if
%
\begin{equation}
\label{eqpositivity-with-constants} r=\sup_{\mathsf{g}\in\mathsf{G},
\mathsf{g}\leq\chi}\Phi(\mathsf{g}) =\sup
_{\mathsf{g}\in\mathsf{G}\setminus\mathbb{R}}\inf_{X\in
\mathcal{X}} \bigl[\chi(X)-\mathsf{g}(X)
\bigr]+\Phi(\mathsf{g})<\infty.
\end{equation}
\end{proposition}

\begin{pf}
Since every element of $\mathsf{G}$ can be written $c+\mathsf{g}$
with $\mathsf{g}\in
\mathsf{G}\setminus\mathbb{R}$ and $c\in\mathbb{R}$, the left-hand
side of
(\ref{eqpositivity-with-constants}) equals
\[
r=\sup_{\mathsf{g}\in\mathsf{G}\setminus\mathbb{R}}\sup_{c\in
\mathbb{R}\dvtx  c+\mathsf{g}\leq
\chi}c +\Phi(\mathsf{g}) =
\sup_{\mathsf{g}\in\mathsf{G}}c_{\mathsf
{g}}+\Phi(\mathsf{g}),
\]
where $c_{\mathsf{g}}=\inf_{X\in\mathcal{X}}(\chi-\mathsf{g})(X)$
is the largest $c$
such that $c+\mathsf{g}\leq\chi$, which yields the equality in
(\ref{eqpositivity-with-constants}).

The necessity of (\ref{eqpositivity-with-constants}) is
straightforward because $r\leq\Phi(\chi)<\infty$. For the
sufficiency, assume that (\ref{eqpositivity-with-constants})
holds. The proof consists in checking that assigning the value
$\Phi(\chi)= r$ yields a positive extension on
$\mathsf{G}+\mathbb{R}\chi$. Let us first prove that $\Phi$ is positive
on $\mathsf{G}$. If some $\mathsf{g}\leq0$ satisfies $\Phi(\mathsf
{g})>0$, then
$\Phi(t\mathsf{g})\uparrow\infty$ as $t\to\infty$ whereas
$t\mathsf{g}\leq
\chi$, which contradicts (\ref{eqpositivity-with-constants}).

Let $\mathsf{g}+\lambda\chi\geq0$ for $\lambda\neq0$ and
$\mathsf{g}\in\mathsf{G}$. If $\lambda>0$, then $-\lambda
^{-1}\mathsf{g}\leq
\chi$, whence $\Phi(-\lambda^{-1}\mathsf{g})\leq r$ and $\Phi
(\mathsf{g}+\lambda
\chi) \geq- \lambda r+\lambda\Phi(\chi)= 0$. If $\lambda<0$,
$-\lambda^{-1}\mathsf{g}\geq\chi$ whence $-\lambda^{-1}\mathsf
{g}$ is larger
than any $\mathsf{g}'\leq\chi$, and
\[
\Phi\bigl(-\lambda^{-1}\mathsf{g}\bigr)\geq\sup_{\mathsf{g}'\in{\mathsf
{G}},\mathsf{g}'\leq\chi}
\Phi\bigl(\mathsf{g}'\bigr)= r
\]
by
monotonicity of $\Phi$ on $\mathsf{G}$. Hence,
$\Phi(\mathsf{g}+\lambda\chi)\geq-\lambda r+\lambda\Phi(\chi)=0$.
\end{pf}

The advantage of the latter condition in
(\ref{eqpositivity-with-constants}) consists in the explicit
reference to the space $\mathcal{X}$ where random elements lie instead
of checking
the inequality $g\leq\chi$.

\subsection{Regularity conditions and distributions of random elements}\label{secregul-cond-distr}

Let $\mathsf{E}$ be a certain family of functions $\mathsf
{v}\dvtx \mathcal{X}\mapsto\mathbb{R}$
defined on a space $\mathcal{X}$ with lattice operation being the pointwise
maximum and the corresponding partial order.

%
\begin{theorem}[(Daniell, see \cite{dud02}, Section~4.5  and \cite{koen97}, Theorem~14.1)]\label{thrdan}
Let a vector lattice $\mathsf{E}$ consist of real-valued functions on
$\mathcal{X}$ and let $\mathsf{E}$ contain constants. If $\Phi$ is a positive
functional on $\mathsf{E}$ such that $\Phi(\mathsf{v}_n)\downarrow
0$ for each
sequence $\mathsf{v}_n\downarrow0$ and $\Phi(1)=1$, then there
exists a
unique random element $\xi$ in $\mathcal{X}$, measurable with respect
to the
$\sigma$-algebra generated by all functions from $\mathsf{E}$, such
that $\Phi(\mathsf{v})=\mathbf{E}\mathsf{v}(\xi)$ for all $\mathsf
{v}\in\mathsf{E}$.
\end{theorem}

In view of the positivity of $\Phi$, the condition imposed on $\Phi$
is equivalent to its upper semi-continuity on $\mathsf{E}$.
In this paper, we start with a functional $\Phi$ defined on a vector
sub-space $\mathsf{G}\subset\mathsf{E}$ and discuss the existence of a
random element $\xi\in\mathcal{X}$ such that $\Phi(\mathsf
{g})=\mathbf{E}\mathsf{g}(\xi)$ for all
$\mathsf{g}\in\mathsf{G}$. In this case, $\Phi$ is said to be \emph
{realisable}
as a probability distribution on $\mathcal{X}$.

%
\begin{assumption}
\label{assumpg}
The vector space $\mathsf{G}$ of functions on $\mathcal{X}$ contains constants
and, for each $\mathsf{g}_1,\mathsf{g}_2\in\mathsf{G}$, there
exists a $\mathsf{g}\in
\mathsf{G}$ such that $(\mathsf{g}_{1}\vee\mathsf{g}_{2})\leq
\mathsf{g}$.
\end{assumption}

From now on assume that $\mathcal{X}$ is a completely regular topological
space, that is, each closed set and each singleton disjoint from it can be
separated by a continuous function.

%
\begin{definition}
\label{defreg-mod}
Given a vector space $\mathsf{G}$ of functions on $\mathcal{X}$, a
\emph{regularity modulus} on $\mathcal{X}$
is a lower semi-continuous function
$\chi\dvtx \mathcal{X}\mapsto[0,\infty]$ such that
%
\begin{equation}
\label{eqhlambda} \mathcal{H}_\mathsf{g}=\bigl\{X\in\mathcal{X}\dvtx  \chi(X)
\leq\mathsf{g}(X)\bigr\}
\end{equation}
is relatively compact for each $\mathsf{g}\in\mathsf{G}$ (if all
$\mathsf{g}\in
\mathsf{G}$ are bounded, $\chi$ is a regularity modulus if and only if
it has compact level sets).
\end{definition}

Examples of regularity moduli are given in
Sections~\ref{secpoint-processes} and \ref{secreals-under-smoothn}.
A measurable function $\mathsf{v}\dvtx \mathcal{X}\mapsto\mathbb{R}$ is
said to be
\emph{$\chi$-regular} if $\mathsf{v}$ is continuous on $\mathcal
{H}_\mathsf{g}$ for each
$\mathsf{g}$ in $\mathsf{G}$. Each continuous function is trivially
$\chi$-regular. The proof of the following central result is based on
the ideas from the proof of \cite{kunlebspeer11}, Theorem~3.14. It
should be noted that our result entails not only the realisability,
but also provides a bound for the expected value of the regularity
modulus. It also holds on not necessarily completely regular space
$\mathcal{X}$ if the regularity modulus is continuous or otherwise
without the
explicit bound on $\mathbf{E}\chi(\xi)$.

%
\begin{theorem}
\label{thrreal-reg}
Consider a vector space $\mathsf{G}$ of functions on $\mathcal{X}$ satisfying
Assumption~\ref{assumpg} and such that each $\mathsf{g}$ from
$\mathsf{G}$ is
$\chi$-regular for a regularity modulus $\chi$. Let $\Phi$ be a
linear functional on $\mathsf{G}$ with $\Phi(1)=1$. Then, for any given
$r\geq0$, there exists a Borel random element $\xi$ in $\mathcal{X}$ such
that
%
\begin{equation}
\label{eqf-xi} %
\cases{ \mathbf{E}\mathsf{g}(\xi)=\Phi(\mathsf{g}) &\quad for all $\mathsf{g}\in\mathsf{G}$,
\vspace*{3pt}\cr
\mathbf{E}\chi(\xi)\leq r,}
\end{equation}
if and only if
%
\begin{equation}
\label{eqlip} \sup_{\mathsf{g}\in\mathsf{G}, \mathsf{g}\leq\chi}\Phi
(\mathsf{g})\leq r.
\end{equation}
\end{theorem}

\begin{pf}
Condition (\ref{eqlip}) is necessary because $\mathsf{g}\leq\chi$ implies
$\Phi(\mathsf{g})=\break \mathbf{E}\mathsf{g}(\xi)\leq\mathbf{E}\chi
(\xi)\leq r$.

\textit{Sufficiency}.
Let $\mathsf{E}$ be the family of all $\chi$-regular functions
$\mathsf{v}$
that satisfy $\mathsf{v}\leq\mathsf{g}$ for some $\mathsf{g}\in
\mathsf{G}$. Each function
$\mathsf{v}\in\mathsf{E}$ is Borel measurable. Note that $\mathsf
{E}$ contains
all bounded continuous functions that generate the Baire
$\sigma$-algebra on $\mathcal{X}$ being in general a sub-$\sigma
$-algebra of
the Borel one. For each $\mathsf{v}_{1},\mathsf{v}_{2}\in\mathsf
{E}$, the function
$\mathsf{v}_{1}\vee\mathsf{v}_{2}$ is $\chi$-regular and is
majorised by
$\mathsf{g}_1\vee\mathsf{g}_2$, where $\mathsf{g}_1,\mathsf
{g}_2\in\mathsf{G}$ majorise $\mathsf{v}_1$
and $\mathsf{v}_2$, respectively. In view of Assumption~\ref{assumpg},
$\mathsf{E}$ is a lattice.

Without loss of generality, assume that the supremum in
(\ref{eqlip}) equals $r$. By Proposition~\ref{corsingle-ext},
$\Phi$ is positive on $\mathsf{G}$ and can be positively extended onto
$\mathsf{G}+\mathbb{R}\chi$ with $\Phi(\chi)=r$, and further on to
$\mathsf{E}+\mathbb{R}\chi$ by Theorem~\ref{thrext}. It remains to
prove that
the obtained extension satisfies conditions of
Theorem~\ref{thrdan}. For that, we use an argument similar to that
of \cite{kunlebspeer11}. First, restrict the obtained functional
$\Phi$ onto~$\mathsf{E}$. Assume that $\chi$ is strictly positive.
Consider a sequence $\{\mathsf{v}_n, n\geq1\}\subset\mathsf{E}$
such that
$\mathsf{v}_n\downarrow0$. For each $n$, let $\mathsf{g}_n$ be a
function of
$\mathsf{G}$ such that $\mathsf{v}_n\leq\mathsf{g}_n$. Take
$\varepsilon>0$. Then
$\mathcal{K}_n=\{X\dvtx  \mathsf{v}_n(X)\geq\varepsilon\chi(X)\}$ is a
subset of
relatively compact $\mathcal{H}_{\mathsf{g}_n/\varepsilon}$, since
$\chi$ is a regularity
modulus. Since $\mathsf{v}_n$ is continuous on $\mathcal{H}_{\mathsf
{g}_n/\varepsilon}$, the set
$\mathcal{K}_n$ is closed and, therefore, compact. The pointwise convergence
$\mathsf{v}_n\downarrow0$ yields that $\bigcap_n \mathcal
{K}_n=\varnothing$ (recall that
$\chi$ is strictly positive). Since $\{\mathcal{K}_n\}$ is a decreasing
sequence of compact sets, $\mathcal{K}_{n_0}=\varnothing$ for some $n_0$,
whence $\mathsf{v}_n(X)< \varepsilon\chi(X)$ for sufficiently large
$n$. The
positivity of $\Phi$ on $\mathsf{E}+\mathbb{R}\chi$ implies $\Phi
(\mathsf{v}_n)\leq
\varepsilon\Phi(\chi)=\varepsilon r$, whence $\Phi(\mathsf
{v}_n)\downarrow0$.
Theorem~\ref{thrdan} yields the existence of a random element $\xi$
in $\mathcal{X}$ such that $\Phi(\mathsf{v})=\mathbf{E}\mathsf
{v}(\xi)$ for all $\mathsf{v}\in\mathsf{E}$.

Since $\chi$ is lower semi-continuous and $\mathcal{X}$ is completely
regular, it can be pointwisely approximated from below by a sequence
$\{\mathsf{v}_n\}$ of nonnegative continuous functions; see
\cite{bour89}, Chapter~9. Then $\tilde{\mathsf{v}}_n=\min(n,\mathsf
{v}_n)$ belongs to
$\mathsf{E}$ and also approximates $\chi$ from below, so that
$\mathbf{E}
\tilde{\mathsf{v}}_n(\xi)=\Phi(\tilde{\mathsf{v}}_n)\leq\Phi
(\chi)=r$, while the
monotone convergence theorem yields
\[
\mathbf{E}\chi(\xi) =\lim_{n\to\infty}\mathbf{E}\tilde{
\mathsf{v}}_n(\xi)\leq r.
\]
If $\chi$ is not strictly positive, it suffices to apply the above
argument to $\chi'=1+\chi$ and use the linearity of $\Phi$.
\end{pf}

Condition (\ref{eqlip}), equivalent to
(\ref{eqpositivity-with-constants}), is expressed solely in terms of
the values taken by $\Phi$ on $\mathsf{G}$ and, therefore, yields a
self-contained solution of the realisability problem. It is not easy
to check in general, but if $\chi$ can be approximated by functions
$\chi_{n}\in\mathsf{G}$, $n\geq1$, then it is possible to ``split''
(\ref{eqlip}) into the positivity condition on $\Phi$ and the uniform
boundedness of $\Phi(\chi_n)$, $n\geq1$. This idea is used
successfully in several different frameworks, which justify the
abstract setting of Theorem~\ref{thrreal-reg}: in
Section~\ref{secpoint-processes} for point processes (see
Theorem~\ref{thmrealisability-intro}), in
Section~\ref{secsmoothm-restr} for random closed sets (see
Theorem~\ref{thmset-bounded-variation}) and in \cite{GalLac} in the
framework of random measurable sets with the regularity modulus being
the perimeter of a set.

The realisability problem is particularly simple if $\mathcal{X}$ is compact
and $\mathsf{G}$ consists of continuous functions. Then, for identically
vanishing $\chi$, Theorem~\ref{thrreal-reg} yields the following
result, which is similar to the Riesz--Markov theorem; see
\cite{koen97}.

%
\begin{corollary}
\label{correal-reg}
Let $\mathcal{X}$ be a compact space with its Borel $\sigma$-algebra.
Consider a vector space $\mathsf{G}$ containing constants such that
each $\mathsf{g}\in\mathsf{G}$ is continuous and a map $\Phi
\dvtx \mathsf{G}\mapsto\mathbb{R}$
such that $\Phi(1)=1$. Then there exists a random element $\xi$ in
$\mathcal{X}$ such that $\mathbf{E}\mathsf{g}(\xi)=\Phi(\mathsf
{g})$ for all $\mathsf{g}\in\mathsf{G}$ if and
only if $\Phi$ is a linear positive functional on $\mathsf{G}$.
\end{corollary}

It should be noted that the complete regularity assumption on $\mathcal
{X}$ is
not needed if the regularity modulus $\chi$ is continuous.

\subsection{Passing to the limit}\label{secpassing-limit}

The following result shows that the family of all random elements that
realise $\Phi$ in the sense of (\ref{eqf-xi}) is weakly compact.

%
\begin{theorem}
\label{thrclosedness}
Assume that $\mathsf{G}$ satisfies Assumption~\ref{assumpg} and
consists of continuous functions on a Polish space $\mathcal{X}$ with
regularity modulus $\chi$. Let $\Phi$ be a linear positive
functional on $\mathsf{G}$. Then the family $\mathfrak{M}$ of all
Borel random
elements $\xi$ that satisfy (\ref{eqf-xi}) for any given $r\geq0$
is compact in the weak topology.
\end{theorem}

\begin{pf}
Since $\chi$ is a regularity modulus, the set $\mathcal
{H}_{r/\varepsilon}$ is
compact. By Markov's inequality,
\[
\mathbf{P}\{\xi\notin\mathcal{H}_{r/\varepsilon}\}=\mathbf{P}\bigl\{
\chi(\xi)>r/
\varepsilon\bigr\}\leq\varepsilon,
\]
for all $\xi\in\mathfrak{M}$, so that $\mathfrak{M}$ is tight.

Let $\{\xi_n, n\geq1\}$ be random elements from $\mathfrak{M}$.
Assume that
$\xi_n$ converges weakly to some $\xi$. Without loss of generality,
assume that the $\xi_n$'s are defined on the same probability space
and converge almost surely to $\xi$. Since $\chi$ is nonnegative,
Fatou's lemma yields
\[
r\geq\liminf\mathbf{E}\chi(\xi_n) \geq\mathbf{E}\liminf\chi(
\xi_n) \geq\mathbf{E}\chi(\lim\xi_n)=\mathbf{E}\chi(\xi),
\]
where the lower semi-continuity of $\chi$ also has been used.

Take an arbitrary $\mathsf{g}\in\mathsf{G}$ and define $\mathcal
{H}_{\lambda\mathsf{g}}$ as
in (\ref{eqhlambda}). Let $\mathsf{g}^+(X)=\max(\mathsf{g}(X),0)$
be the
positive part of $\mathsf{g}$. Then, for $\lambda>0$,
\[
\mathbf{E}\mathsf{g}^+(\xi_n)=\mathbf{E}\mathsf{g}^+(\xi
_n)\mathbh{1}_{\xi_n \notin\mathcal{H}_{\lambda\mathsf{g}}} + \mathbf
{E}\mathsf{g}^+(
\xi_n)\mathbh{1}_{\xi_n \in\mathcal
{H}_{\lambda\mathsf{g}}}.
\]
Since $\mathsf{g}$ is continuous, $\mathcal{H}_{\lambda\mathsf{g}}$
is closed (and compact), so
that if $\xi_n\in\mathcal{H}_{\lambda\mathsf{g}}$ for infinitely
many $n$, then also
$\xi\in\mathcal{H}_{\lambda\mathsf{g}}$. Furthermore, $\lambda
\mathsf{g}$ and also $\mathsf{g}$
itself, are continuous and bounded on $\mathcal{H}_{\lambda\mathsf
{g}}$, so that
Fatou's lemma yields
\begin{eqnarray*}
\limsup\mathbf{E}\mathsf{g}^+(\xi_n)\mathbh{1}_{\xi_n \in
\mathcal{H}_{\lambda\mathsf{g}}} &\leq&
\mathbf{E}\limsup\bigl(\mathsf{g}^+(\xi_n)\mathbh{1}_{\xi_n \in
\mathcal{H}_{\lambda\mathsf{g}}}
\bigr)
\\
&\leq&\mathbf{E}\mathsf{g}^+(\xi)\mathbh{1}_{\xi\in\mathcal
{H}_{\lambda\mathsf{g}}}\leq\mathbf{E}
\mathsf{g}^+(\xi).
\end{eqnarray*}
Thus,
\[
\limsup\mathbf{E}\mathsf{g}^+(\xi_n)\leq\mathbf{E}\frac{\chi
(\xi_n)}{\lambda}+
\mathbf{E}\mathsf{g}^+(\xi) \leq\frac{r}{\lambda}+\mathbf{E}\mathsf
{g}^+(\xi).
\]
Since $\lambda$ is arbitrary,
\[
\limsup\mathbf{E}\mathsf{g}^+(\xi_n)\leq\mathbf{E}\mathsf{g}^+(\xi).
\]
Since $\mathsf{g}^+$ is nonnegative, Fatou's lemma yields that
$\mathbf{E}
\mathsf{g}^+(\xi_n)\to\mathbf{E}\mathsf{g}^+(\xi)$. By applying
the same argument to
the function $(-\mathsf{g})$, $\lim\mathbf{E}\mathsf{g}(\xi
_n)=\mathbf{E}\mathsf{g}(\xi)$, so that $\mathbf{E}
\mathsf{g}(\xi)=\Phi(\mathsf{g})$ for all $\mathsf{g}\in\mathsf
{G}$. Therefore,
$\xi\in\mathfrak{M}$.
\end{pf}

The following result concerns realisability of pointwise limits of
linear functionals. Special conditions of this type for correlation
measures of point processes are given in
\cite{kunlebspeer11}, Section~3.4.

%
\begin{theorem}
\label{thrlimits}
Let $\{\Phi_n, n\geq1\}$ be a sequence of linear positive functionals
on a space $\mathsf{G}$ that satisfies the assumptions of
Theorem~\ref{thrclosedness}. Assume that
%
\begin{equation}
\label{eqliminfn-supgin-g} \liminf_n \sup_{\mathsf{g}\in\mathsf{G},
\mathsf{g}\leq\chi}
\Phi_n(\mathsf{g})<\infty.
\end{equation}
If $\Phi_n(\mathsf{g})\to\Phi(\mathsf{g})$ for all $\mathsf{g}\in
\mathsf{G}$, then $\Phi$
is realisable as a random element $\xi$ satisfying (\ref{eqf-xi})
and such that $\xi$ is the weak limit of random elements realising
$\Phi_{n_k}$ for a subsequence $n_k$.
\end{theorem}

\begin{pf}
By passing to a subsequence, it suffices to assume that
(\ref{eqliminfn-supgin-g}) holds for the limit instead of the
lower limit. Let $\xi_n$ be a random element that realises $\Phi_n$.
If $r$ is larger than the limit of (\ref{eqliminfn-supgin-g}), then
$\mathbf{P}\{\xi_n\notin\mathcal{H}_{r/\varepsilon}\}\leq
\varepsilon$, so that
$\{\xi_n\}$ is a tight sequence. Without loss of generality, assume
that $\xi_n$ weakly converges to a random element $\xi$.

The pointwise convergence of $\Phi_n$ yields that $\mathbf{E}\mathsf
{g}(\xi_n)\to
\Phi(\mathsf{g})$ for all $\mathsf{g}\in\mathsf{G}$. Now the
arguments from the
proof of Theorem~\ref{thrclosedness} can be used to show that
$\mathbf{E}
\mathsf{g}(\xi_n)\to\mathbf{E}\mathsf{g}(\xi)$, so that $\mathbf
{E}\mathsf{g}(\xi)=\Phi(\mathsf{g})$ for all
$\mathsf{g}\in\mathsf{G}$, that is, $\xi$ indeed satisfies (\ref
{eqf-xi}).
\end{pf}

\subsection{Invariant extension}\label{secinvariant-extensions}

Consider an Abelian group $\Theta$ of continuous transformations
acting on $\mathcal{X}$. For a function $\mathsf{v}$ on $\mathcal
{X}$, define
\[
(\theta\mathsf{v}) (X)=\mathsf{v}(\theta X),\qquad\theta\in\Theta,X\in
\mathcal{X}.
\]
A functional $\Phi$ is said to be $\Theta$-invariant if, for each
$\theta\in\Theta$ and $\mathsf{v}$ from the domain of definition of
$\Phi$,
$\Phi(\theta\mathsf{v})$ is defined and equal to $\Phi(\mathsf{v})$.

A Borel random element $\xi$ in $\mathcal{X}$ is said to be
\emph{$\Theta$-stationary} if, for each $\theta\in\Theta$, $\theta
\xi$ has the same distribution as $\xi$.
A variant of the following result for correlation measures of point
processes is given in \cite{kunlebspeer11}, Theorem~4.3.

%
\begin{theorem}
\label{thrchi-approx}
\label{thminvar-single}
Assume that $\mathsf{G}$ is a $\Theta$-invariant space satisfying
Assumption~\ref{assumpg} and consisting of $\chi$-regular
functions. Furthermore, assume that at least one of the following
conditions holds:
\begin{longlist}[(ii)]
\item[(i)] $\mathsf{G}$ consists of continuous functions and $\chi$ is
pointwisely approximated from below by a monotone sequence of
functions $\mathsf{g}_n\in\mathsf{G}$, $n\geq1$.
\item[(ii)] $\chi$ is $\Theta$-invariant.
\end{longlist}
Let $\Phi$ be a $\Theta$-invariant functional on $\mathsf{G}$. Then,
for every given $r\geq0$, there exists a $\Theta$-stationary random
element $\xi$ in $\mathcal{X}$ satisfying (\ref{eqf-xi})
if and only if (\ref{eqlip}) holds.
\end{theorem}

\begin{pf}
(i)~As in \cite{kunlebspeer11}, Proposition 4.1, the proof consists in
checking hypotheses of the Markov--Kakutani fixed-point theorem.
Let $\mathfrak{M}$ be the family of random elements $\xi$ that
realise $\Phi$
on $\mathsf{G}$, and satisfy $\mathbf{E}\chi(\theta\xi)\leq r$ for every
$\theta\in\Theta$. The family $\mathfrak{M}$ is easily seen to be convex
with respect to addition of measures, it is compact by
Theorem~\ref{thrclosedness}, and $\Theta$-invariant, since $\Phi$
is $\Theta$-invariant on $\mathsf{G}$. It remains to prove that
$\mathfrak{M}$
is not empty.

In view of (\ref{eqlip}), it is possible to extend
$\Phi$ positively onto $\mathsf{G}+\mathbb{R}\chi$, so that
$\mathbf{E}\chi(\xi)\leq
r$. The $\Theta$-invariance of $\Phi$ on $\mathsf{G}$ together with the
monotone convergence theorem imply that $\mathbf{E}\chi(\theta\xi)
=\mathbf{E}
\chi(\xi)\leq r$, whence $\xi\in\mathfrak{M}$.

(ii)~By Proposition~\ref{corsingle-ext}, we can extend
$\Phi$ positively onto the $\Theta$-invariant vector space
$\mathsf{V}=\mathsf{G}+\mathbb{R}\chi$. Since $\Phi$ is $\Theta
$-invariant on
$\mathsf{G}$, we have $\Phi(\theta(\mathsf{g}+t\chi))=\Phi(\theta
\mathsf{g})+t
\Phi(\theta\chi)=\Phi(\mathsf{g}+t \chi)$ for $\mathsf{g}+t\chi
$ in $\mathsf{V}$,
whence $\Phi$ is $\Theta$-invariant on $\mathsf{V}$. According to
\cite{sil56}, Theorem~3, $\Phi$ admits a positive $\Theta$-invariant
extension to the space $\mathsf{E}+\mathbb{R}\chi$, defined like in
the proof
of Theorem~\ref{thrreal-reg}. The restriction of the obtained
functional onto $\mathsf{E}$ corresponds to a random element $\xi$ in
$\mathcal{X}$ that verifies (\ref{eqf-xi}) and satisfies $\mathbf
{E}(\theta
\mathsf{v})(\xi)=\Phi(\theta\mathsf{v})=\Phi(\mathsf{v})=\mathbf
{E}\mathsf{v}(\xi)$,
$\theta\in\Theta$, for $\mathsf{v}$ in $\mathsf{E}$. Since
$\mathsf{E}$ contains
all bounded continuous functions on $\mathcal{X}$, $\theta\xi$ and
$\xi$
are identically distributed for all $\theta\in\Theta$.
\end{pf}

\section{Correlation measures of point processes}\label{secpoint-processes}

\subsection{Framework and main results}\label{secmoment-conditions}

Let $\mathcal{N}$ be the family of locally finite counting measures on a
locally compact complete separable metric space $\mathbb{X}$. We denote
the support of $Y\in\mathcal{N}$ by the same letter $Y$, so that
$x\in Y$
means $Y(\{x\})\geq1$.

Equip $\mathcal{N}$ with the vague topology, see \cite{davj},
Chapter~7, so that
$\mathcal{N}$ is metric and so completely regular. A random element
$\xi$ in
$\mathcal{N}$ with the corresponding Borel $\sigma$-algebra is called a
\emph{point process}. Denote by $\mathcal{N}_0$ the family of
\emph{simple} counting measures, that is, those which do not attach mass
2 or more to any given point. If $\xi$ is simple, that is, $\xi\in
\mathcal{N}_0$ a.s., then $\xi$ can be identified with a locally
finite random set in $\mathbb{X}$, which is also denoted by $\xi$.

For a real function $h$ on $\mathbb{X}\times\mathbb{X}$ and counting measure
$Y=\sum_i\delta_{x_i}$ given by the sum of Dirac measures, define
\[
\mathsf{g}_h(Y)=\sum_{x_i,x_j\in Y, i\neq j}h(x_i,x_j),
\]
whenever the series absolutely converges, the empty sum being
$0$. Note that the sum in the right-hand side is taken over all pairs
of distinct points from the support of $Y$, where multiple points
appear several times according to their multiplicities. The value
$\mathsf{g}_h(Y)$ is necessarily finite if $h$ is bounded and has a bounded
support. The value $\mathsf{g}_h(Y)$ is termed in \cite
{kunlebspeer11} the
quadratic polynomial of $Y$, while polynomials of order $n\geq1$ are
sums of functions of $n$ points of the process, and are constants if
$n=0$.

Let $\mathsf{G}$ be the vector space formed by constants and functions
$\mathsf{g}_h$ for $h$ from the space $\mathscr{C}_{\mathrm{o}}$ of
symmetric continuous functions
with compact support. Note that $\mathsf{G}$ satisfies
Assumption~\ref{assumpg}, since
\[
(c_{1}+\mathsf{g}_{h_{1}})\vee(c_{2}+
\mathsf{g}_{h_{2}}) \leq c_{1}\vee c_{2}+
\mathsf{g}_{h_{1}\vee h_{2}}\in\mathsf{G}
\]
for all $c_1,c_2\in\mathbb{R}$ and $h_1,h_2\in\mathscr{C}_{\mathrm
{o}}$. Furthermore, each $\mathsf{g}_h$
is continuous in the vague topology, and so is $\chi$-regular for any
regularity modulus $\chi$.

Assume that $\xi$ has locally finite second moment, that is, $\mathbf
{E}\xi(A)^2$
is finite for each bounded $A$. The \emph{correlation measure} $\rho$
(also called the second factorial moment measure) of a point process
$\xi$ is a measure on $\mathbb{X}\times\mathbb{X}$ that satisfies
%
\begin{equation}
\label{eqpp-realisable} \int_{\mathbb{X}\times\mathbb{X}}h(x,y)\rho
(dx\,dy)=\mathbf{E}
\mathsf{g}_h(\xi)
\end{equation}
for each $h\in\mathscr{C}_{\mathrm{o}}$; see \cite{davj}, Section~5.4 and
\cite{skm}, Section~4.3.
The left-hand side defines a linear functional $\Phi(\mathsf{g}_h)$ on
$\mathsf{g}_h\in\mathsf{G}$.

Let $\mathcal{X}$ be a subset of $\mathcal{N}$, which may be
$\mathcal{N}$ itself. Recall that
a subset of a completely regular space is completely regular, see
\cite{kur66}, Theorem~14.I.2. Given a symmetric locally finite measure
$\rho$ on $\mathbb{X}\times\mathbb{X}$, the \emph{realisability problem}
amounts to the existence of a point process $\xi$ with realisations
from $\mathcal{X}$ and with correlation measure $\rho$, so that
$\Phi(\mathsf{g}_h)=\mathbf{E}\mathsf{g}_h(\xi)$ for all $h\in
\mathscr{C}_{\mathrm{o}}$.

By (\ref{eqphi-positivity}), the positivity of $\Phi$ means
%
\begin{equation}
\label{eqppk} \Phi(\mathsf{g}_h)\geq\inf_{Y\in\mathcal{X}}
\mathsf{g}_h(Y)
\end{equation}
for all $h\in\mathscr{C}_{\mathrm{o}}$. Then it is clear that the
positivity of $\Phi$ is
necessary for its realisability. If $\mathcal{X}$ is compact in the vague
topology, then Corollary~\ref{correal-reg} applies and the positivity
condition (\ref{eqppk}) is necessary and sufficient for the
realisability of $\rho$.

However, in general the positivity condition alone is not sufficient
for the realisability; see \cite{kunlebspeer07},
Example~3.12. In the
following, we find another condition that is not directly related to
the positivity, but together with the positivity, is necessary and
sufficient for the realisability.

As an introduction, let us present our results for $\mathbb{X}$ being a
subset of the Euclidean space $\mathbb{R}^{d}$. For $\varepsilon\geq
0$, define
\[
\chi_{\varepsilon}(Y) =\sum_{x,y\in Y, x\neq
y}\|x-y
\|^{-d-\varepsilon}, \qquad Y\in\mathcal{N},
\]
which is later acknowledged as being a regularity modulus (see
Definition \ref{defreg-mod}) if $\varepsilon\neq0$. Note that
$\chi_\varepsilon(Y)$ is infinite if $Y$ has multiple points.
The tools developed in this paper enable us to resolve the original
realisability problem with a supplementary regularity condition
involving $\chi_{\varepsilon}$.

%
\begin{theorem}
\label{thmrealisability-intro}
\textup{(i)} Let $\mathbb{X}$ be a compact subset of $\mathbb
{R}^{d}$ without
isolated points. A symmetric finite measure $\rho(dx\,dy)$ on
$\mathbb{X}\times\mathbb{X}$ is the correlation measure of a simple
point process
$\xi\subset\mathbb{X}$ such that $\mathbf{E}\chi_{0}(\xi)<\infty
$ if and only
if $\Phi$ given by the left-hand side of (\ref{eqpp-realisable}) is
positive and
\[
\int_{\mathbb{X}\times\mathbb{X}} \|x-y\|^{-d}\rho(dx\,dy)<\infty.
\]

\textup{(ii)} Let $\rho$ be a symmetric locally finite measure on
$\mathbb{R}^{d}\times\mathbb{R}^{d}$ such that \mbox{$\rho((A+x)\times
(B+x))=\rho(A\times
B)$} for all $x\in\mathbb{R}^d$ and measurable sets $A$ and $B$. Then there
exists a simple stationary point process $\xi$ with correlation
measure $\rho$, such that
\[
\mathbf{E}\chi_{0}(\xi\cap C)<\infty
\]
for every compact $C\subset\mathbb{R}^{d}$, if and only if $\Phi$
defined by
(\ref{eqpp-realisable}) is positive and
%
\begin{equation}
\label{eqreg-condition-intro} \int_{B\times B}\|x-y\|^{-d}\rho(dx\,dy)<
\infty
\end{equation}
for some open set $B$.
\end{theorem}

\begin{pf}
The proof relies of several theorems that will appear later in this
section. The first statement follows from
Theorem~\ref{thrreal-hardcore-pp} using the fact that the packing
number $P_{t}(\mathbb{X})$ of $\mathbb{X}$ is bounded by $ct^{-d}$
for all
sufficiently small $t$. For \textup{(ii)}, apply
Theorem~\ref{thstat}\textup{(ii)} noticing that the imposed
condition is equivalent to (\ref{eqrhobar-regular}).
\end{pf}

In the remainder of this section, one can find a quantification of
this result [i.e., how the left-hand member of
(\ref{eqreg-condition-intro}) controls the value of $\mathbf{E}
\chi_{0}(X\cap C)]$ as well as generalisations for general metric
spaces. The main argument used is a splitting method based on
Theorem~\ref{thrreal-reg}; the details are made clear in the proof of
Theorem~\ref{thrreal-pp-psi}. Note that the packing number of the
metric space appears as a crucial quantity to uncouple in this way the
realisability problem; see Lemma~\ref{thrpsi-rm}.

\subsection{Moment conditions}

The family $\mathcal{X}_k$ of all counting measures with total mass at most
$k$ on a compact space $\mathbb{X}$ is compact. Thus, a measure $\rho$
on $\mathbb{X}\times\mathbb{X}$ is realisable as a point process
with at most $k$
points if (\ref{eqppk}) holds with $\mathcal{X}=\mathcal{X}_k$.

Assume that $Y$ is a finite counting measure. For $\alpha>2$, define
\[
\chi_{\alpha}(Y)=Y(\mathbb{X})^\alpha,\qquad Y\in\mathcal{N}.
\]
The finiteness of $\mathbf{E}\chi_\alpha(\xi)$ amounts to the
finiteness of
the moment of order $\alpha$ for the total mass of $\xi$. Since
$h\in\mathscr{C}_{\mathrm{o}}$ is bounded by a constant $c'$ and
$\alpha>2$, the family
\[
\bigl\{Y\in\mathcal{N}\dvtx  \chi_\alpha(Y)\leq c+\mathsf{g}_h(Y)
\bigr\} \subset\bigl\{Y\in\mathcal{N}\dvtx  Y(\mathbb{X})^\alpha\leq
c+c'Y(\mathbb{X})^2\bigr\}
\]
consists of counting measures with total masses bounded by a certain
constant and, therefore, is compact in the space $\mathcal{N}$. Hence,
$\chi_{\alpha}$ is a regularity modulus and so
Theorem~\ref{thrreal-reg} yields the realisability condition
%
\begin{equation}
\label{eqsupgin-g--2-a} \sup_{\mathsf{g}\in\mathsf{G}, \mathsf{g}\leq
\chi_\alpha} \Phi(\mathsf{g})<\infty
\end{equation}
of $\rho$ by a point process $\xi$ whose total number of points has
finite moment of order~$\alpha$. Note that
\cite{kunlebspeer11}, Theorem~3.14, provides a variant of this result
assuming the existence of the third factorial moment of the
cardinality of $\xi$ (i.e., with $\alpha=3$) and for the joint
realisability of the intensity and the correlation measures. The
condition of \cite{kunlebspeer11}, Theorem~3.14 (reformulated for the
correlation measure only) reads in our notation as
$c+\Phi(\mathsf{g}_h)+br\geq0$ whenever $c+\mathsf{g}_h+b\chi_3$
is nonnegative on~$\mathcal{N}$. Noticing that $b\geq0$, this is equivalent to the fact that
$c+\Phi(\mathsf{g}_h)\leq r$ whenever $c+\mathsf{g}_h\leq\chi_3$,
being exactly
(\ref{eqsupgin-g--2-a}). If $\Theta$ is a group of continuous
transformations acting on $\mathbb{X}$ and $\rho$ is $\Theta$-invariant,
then the point process $\xi$ can be chosen $\Theta$-stationary by
Theorem~\ref{thminvar-single}(ii).

In order to handle possibly nonfinite point processes $\xi$, define
\[
\chi_{\alpha,\beta}(Y)= \biggl(\sum_{x\in
Y}\beta(x)
\biggr)^\alpha,\qquad Y\in\mathcal{N},
\]
for a lower semi-continuous strictly positive function
$\beta\dvtx \mathbb{X}\mapsto\mathbb{R}$ and $\alpha>2$. By
approximating $\beta$
from below with compactly supported functions, it is easy to see that
$\chi_{\alpha,\beta}$ is a regularity modulus. By
Theorem~\ref{thrreal-reg} and (\ref{eqpositivity-with-constants}),
for any given $r\geq0$, there is a point process $\xi$ with
correlation measure $\rho$ such that $\mathbf{E}\chi_{\alpha,\beta
}(\xi)\leq
r$ if and only if $\rho$ satisfies
%
\begin{equation}
\label{eqreal-chi-alpha} \inf_{Y\in\mathcal{X}} \bigl[\chi_{\alpha,\beta
}(Y)-\mathsf
{g}_h(Y) \bigr] +\int_{\mathbb{X}\times\mathbb{X}}h(x,y)\rho(dx\,dy)\leq
r,\qquad h\in\mathscr{C}_{\mathrm{o}}.
\end{equation}

For $\alpha=3$, condition (\ref{eqreal-chi-alpha}) is a reformulation
of \cite{kunlebspeer11}, Theorem~3.17, meaning the positivity of
$\Phi$
on a family of positive polynomials that involve symmetric functions
of the support points up to the third order.
The realisability condition for $\Theta$-stationary random elements
can be obtained by applying Theorem~\ref{thrchi-approx}.

\subsection{Hardcore point processes on a compact space}\label{sechardcore}

Assume that $\mathbb{X}$ is a compact metric space with metric
$\mathbf{d}$. Let
$\mathcal{N}_{\varepsilon}$ be the family of \emph{$\varepsilon
$-hard-core} point sets in~$\mathbb{X}$ (including the empty set), that is, each $Y\in\mathcal
{N}_{\varepsilon}$ attaches
unit masses to distinct points with pairwise distances at least
$\varepsilon$
with a fixed $\varepsilon>0$. In this case, no multiple points are allowed,
that is, $\mathcal{N}_{\varepsilon}\subset\mathcal{N}_0$.

According to \cite{holstr78,konkut06}, a subset $\mathcal{X}$ of simple
counting measures $\mathcal{N}_0$ is relatively compact if and only if
$\sup\{Y(K)\dvtx  Y\in\mathcal{X}\}$ is finite and the infimum over
$Y\in\mathcal{X}$ of
the minimal distance between two points in $Y\cap K$ is strictly
positive for each compact set $K\subset\mathbb{X}$. The hard-core
condition yields that the number of points in any compact set is
uniformly bounded, and so $\mathcal{N}_{\varepsilon}$ is indeed
compact. By
Corollary~\ref{correal-reg}, $\rho$ is realisable as the correlation
measure of an $\varepsilon$-hard-core point process with given
$\varepsilon>0$ if
and only if
%
\begin{equation}
\label{eqfghg-infy-ghy} \Phi(\mathsf{g}_h)\geq\inf_{Y\in\mathcal
{N}_{\varepsilon}}
\mathsf{g}_h(Y)
\end{equation}
for all $h\in\mathscr{C}_{\mathrm{o}}$. This result is formulated in
\cite{kunlebspeer11}, Theorem~3.4, which essentially reduces to the
positivity of $\Phi$ over the family $c+\mathsf{g}_h$ (in our setting).

In this paper, we assume that the hardcore distance is not
predetermined and the point process takes realisations from
$\bigcup_{\varepsilon>0}\mathcal{N}_{\varepsilon}$, which coincides
with $\mathcal{N}_0$ in case
of compact $\mathcal{X}$. Note that (\ref{eqfghg-infy-ghy}) is stronger
than the positivity of $\Phi$ on functions $\mathsf{g}_h$ defined on the
whole family $\mathcal{N}_0$ and formulated as
%
\begin{equation}
\label{eqtot-posit} \Phi(\mathsf{g}_h)\geq\inf_{Y\in\mathcal{N}_0}
\mathsf{g}_h(Y),\qquad h\in\mathscr{C}_{\mathrm{o}}.
\end{equation}
If $\mathbb{X}$ does not have isolated points, then the infimum in
(\ref{eqtot-posit}) can be taken over $\mathcal{N}$. This is seen by
approximating a multiple atom with a sequence of simple counting
measures supported by points converging to the atom's location.

In the following, we use the (hard-core) regularity modulus of the form
\[
\chi^{\mathrm{hc}}_{\psi}(Y)=\sum_{x_i, x_j\in Y, i\neq j}
\psi\bigl(\mathbf{d}(x_i,x_j)\bigr), \qquad Y\in
\mathcal{N}_0,
\]
where $\psi\dvtx (0,\infty)\mapsto[0,\infty]$ is a monotone decreasing
right-continuous function, such that $\psi(t)\to\infty$ as
$t\downarrow0$. The compactness of $\mathbb{X}$ and the lower
semi-continuity of $\psi$ imply that $\chi^{\mathrm{hc}}_\psi$ is lower
semi-continuous on $\mathcal{N}_0$.
As shown below $\chi^{\mathrm{hc}}_{\psi}$ is a regularity modulus
if $\psi$ grows
sufficiently fast at zero.

Let $P_{t}(\mathbb{X})$ be the \emph{packing number} of $\mathbb
{X}$, that is, the
maximum number of points in $\mathbb{X}$ with pairwise distances
exceeding $t$, see \cite{mati95}, page~78. It is convenient to define
the packing number at $t=0$ as $P_{0}(\mathbb{X})=\infty$ if $\mathbb
{X}$ is
infinite and otherwise let $P_{0}(\mathbb{X})$ be the cardinality of
$\mathbb{X}$.

%
\begin{lemma}
\label{thrpsi-rm}
Function $\chi^{\mathrm{hc}}_{\psi}$ is a regularity modulus on
$\mathcal{N}_0$ if
%
\begin{equation}
\label{condpsi-rm} \psi(t)/P_{t}(\mathbb{X})\to\infty\qquad\mbox{as } t
\downarrow0.
\end{equation}
\end{lemma}

\begin{pf}
In view of the compactness of $\mathbb{X}$, it is possible to bound
$h\in\mathscr{C}_{\mathrm{o}}$ by a constant $\lambda$, so that
$\chi^{\mathrm{hc}}_\psi$ is a regularity
modulus if
\[
\mathcal{H}_{\lambda}=\bigl\{Y\in\mathcal{N}_0\dvtx
\chi^{\mathrm
{hc}}_{\psi}(Y)\leq\lambda Y(\mathbb{X})^2\bigr
\}
\]
is compact in $\mathcal{N}_0$ for each $\lambda>0$. For this, it
suffices to
show that the total mass of all $Y\in\mathcal{H}_{\lambda}$ is
bounded by a
fixed number and $\mathcal{H}_{\lambda}\subset\mathcal
{N}_{\varepsilon}$ for some
$\varepsilon>0$.

Let $\gamma_t(n)$ be the minimal number of pairs $(x_i,x_j)$ with
$i\neq j$, such that $x_i,x_j\in Y$ and $\mathbf{d}(x_i,x_j)\leq t$ over
all counting measures $Y$ of total mass $n$.

Take $t$ such that $\psi(t)/P_{t}(\mathbb{X})>\lambda$. If
$Y(\mathbb{X})\geq
n$, then
\[
\chi^{\mathrm{hc}}_\psi(Y) \geq\sum_{x_i, x_j\in Y, i\neq j}
\psi(t) \mathbh{1}_{\mathbf
{d}(x_i,x_j)\leq
t} \geq\gamma_t(n)\psi(t).
\]
Therefore,
\[
\mathcal{H}_\lambda\subset\bigl\{Y\dvtx  n^{-2}
\gamma_t(n)\psi(t)\leq\lambda\bigr\}
\]
consists of $Y$ with total mass uniformly bounded by fixed number
$n_\lambda$. Indeed, by Lemma~\ref{lemmaalgo},
\[
\lim_{n\to\infty} n^{-2}\gamma_t(n)\geq\lim
_{n\to\infty} n^{-2}n \biggl(\frac{n}{P_{t}(\mathbb{X})}-1 \biggr)
=P_{t}(\mathbb{X})^{-1}.
\]

Choose $\varepsilon>0$ so that $\psi(t)\geq\lambda n_{\lambda}^2$ for
$t\leq\varepsilon$. For $Y\in\mathcal{H}_{\lambda}$ and any
$x_i,x_j\in Y$,
\[
\psi\bigl(\mathbf{d}(x_i,x_j)\bigr)\leq
\chi^{\mathrm{hc}}_{\psi}(Y) \leq\lambda n_{\lambda}^2,
\]
whence $\mathbf{d}(x_i,x_j)\geq\varepsilon$. Thus, $\mathcal
{H}_{\lambda}\subset
\mathcal{N}_{\varepsilon}$, so $\mathcal{H}_\lambda$ is relatively compact.
\end{pf}

The following theorem shows that the realisability condition can be
split into the positivity condition (\ref{eqtot-posit}) on the linear
functional $\Phi$ and the regularity condition (\ref{eqreg-pp}) on
the correlation measure, so that the latter can be easily
checked. Such a split is possible because the regularity modulus
$\chi^{\mathrm{hc}}_\psi$ can be approximated by functions from
$\mathsf{G}$.

%
\begin{theorem}
\label{thrreal-pp-psi}
A locally finite measure $\rho$ on $\mathbb{X}\times\mathbb{X}$ is
the correlation
measure of a simple point process $\xi$ such that $\mathbf{E}
\chi^{\mathrm{hc}}_{\psi}(\xi)\leq r$ for some $r\geq0$ with $\psi
$ satisfying~(\ref{condpsi-rm}) if and only if (\ref{eqtot-posit}) holds and
%
\begin{equation}
\label{eqreg-pp} \int_{\mathbb{X}\times\mathbb{X}}\psi\bigl(\mathbf
{d}(x,y)\bigr)\rho
(dx\,dy)\leq r.
\end{equation}
\end{theorem}

\begin{pf}
\textit{Necessity}. The definition of the correlation measure
implies that
\[
\int_{\mathbb{X}\times\mathbb{X}}\psi\bigl(\mathbf{d}(x,y)\bigr)\rho
(dx\,dy) =
\mathbf{E}\chi^{\mathrm{hc}}_{\psi}(\xi)\leq r.
\]

\textit{Sufficiency}. First assume that $\psi$ only takes finite
values. The proof consists of checking
(\ref{eqpositivity-with-constants}), which is equivalent to
(\ref{eqlip}).

For each family of positive numbers $\{t_{\mathsf{g}}, \mathsf{g}\in
\mathsf{G}\}$,
%
\begin{equation}
\label{eqinterm-ineq} \qquad\sup_{\mathsf{g}\in\mathsf{G}}\inf_{Y\in\mathcal
{N}_0} \bigl[
\chi(Y)-\mathsf{g}(Y) \bigr]+\Phi(\mathsf{g}) \leq\sup_{\mathsf{g}\in
\mathsf{G}}\inf
_{Y\in\mathcal
{N}_{t_{\mathsf{g}}}} \bigl[\chi(Y)-\mathsf{g}(Y) \bigr]+\Phi(\mathsf{g}).
\end{equation}
The crucial step of the proof consists in the careful choice of
$t_{\mathsf{g}}>0$.

Fix $\mathsf{g}\in\mathsf{G}$. For $t>0$, define
$\psi_{t}(s)=\psi(\max(t,s))$, $s\geq0$. Since any $Y\in\mathcal{N}_{t}$
does not contain any two points at distance less than $t$,
$\chi(Y)=\mathsf{g}_{\psi_{t}}(X)$. Therefore,
%
\begin{equation}
\label{eqinterm-ineq2} \inf_{Y\in\mathcal{N}_{t}}\bigl[\chi(Y)-\mathsf
{g}(Y)\bigr] =
\inf_{Y\in\mathcal{N}_{t}}(\mathsf{g}_{\psi_{t}}-\mathsf{g}) (Y).
\end{equation}
Our aim is to prove that
%
\begin{equation}
\label{eqinterm-claim} \inf_{Y\in\mathcal{N}_{t}}(\mathsf{g}_{\psi_{t}}-
\mathsf{g}) (Y) =\inf_{Y\in\mathcal{N}_0}(\mathsf{g}_{\psi_{t}}-
\mathsf{g}) (Y),
\end{equation}
because then, since $\mathsf{g}_{\psi_{t}}\in\mathsf{G}$, the
positivity of
$\Phi$ on $\mathsf{G}$ yields that (\ref{eqinterm-ineq2}) is not
greater than $\Phi(\mathsf{g}_{\psi_{t}}-\mathsf{g})$. Thus,
(\ref{eqinterm-ineq}) is bounded above by
\[
\sup_{\mathsf{g}\in\mathsf{G}}\Phi(\mathsf{g}_{\psi_{t}}-\mathsf
{g})+\Phi(
\mathsf{g}) \leq\sup_{t}\Phi(\mathsf{g}_{\psi_{t}}) = \int
_{\mathbb{X}\times\mathbb{X}}\psi\bigl(\mathbf{d}(x,y)\bigr)\rho(dx\,dy)
\]
by the monotone convergence theorem.

The proof of (\ref{eqinterm-claim}) relies on the proper choice for
$t$ (depending on $\mathsf{g}$). Assume without loss of generality
$\mathsf{g}=\mathsf{g}_{h}$ for $h\in\mathscr{C}_{\mathrm{o}}$
with absolute value bounded by
$\lambda>0$. By (\ref{condpsi-rm}), there exists $t_{0}$ such that
$\psi(t_0)/P_{t_0}(\mathbb{X})\geq\lambda+1$. By Lemma~\ref{lemmaalgo},
there is $n_{0}$ such that for all $Y$ with mass $n\geq n_{0}$, the
number of pairs of points of $Y$ at distance at most $t_{0}$
satisfies
\[
\gamma_{t_{0}}(n) \geq n^{2}\frac{1}{P_{t_{0}}(\mathbb{X})}.
\]
Choose $t\leq t_{0}$ so that $\psi(t)>\lambda n_{0}^{2}$ and
consider any $Y\in\mathcal{N}_0\setminus\mathcal{N}_{t}$. If
$Y(\mathbb{X})\leq
n_{0}$, then
\[
\mathsf{g}_{\psi_{t}}(Y) \geq\psi(t) >\lambda n_{0}^{2}
\geq\mathsf{g}_{h}(Y),
\]
while if $Y(\mathbb{X})>n_{0}$, then
\[
\mathsf{g}_{\psi_{t}}(Y)-\mathsf{g}_{h}(Y) \geq
\mathsf{g}_{\psi_{t_{0}}}(Y)-\mathsf{g}_{h}(Y) \geq
\psi(t_{0})\gamma_{t_{0}}(Y)-\lambda n^{2}>0.
\]
Thus for $Y\notin\mathcal{N}_{t}$, we have
$\mathsf{g}_{\psi_{t}}(Y)-\mathsf{g}_{h}(Y)>0$. Therefore, the
infimum of
\mbox{$\mathsf{g}_{\psi_{t}}-\mathsf{g}_{h}$}, which is nonpositive
because zero is
obtained for $Y=\varnothing$, is reached on $\mathcal{N}_{t}$, and~(\ref{eqinterm-claim}) is proved.

Now assume that $\psi(t)$ is infinite for $t\in[0,\delta)$ and
finite on $(\delta,\infty)$ with $\delta>0$. If $\psi(t)\to\infty$
as $t\downarrow\delta$, then the above arguments apply with
$t_0>\delta$ chosen such that $\psi(t_0)/P_{\delta}(\mathbb
{X})>\lambda$.

Assume that $\psi(\delta)$ is finite. Let $\psi_0(t)$ be a function
satisfying~(\ref{condpsi-rm}) and finite for all $t>0$, for example,
$\psi_{0}(t)=t^{-1}P_{t}(\mathbb{X})$. Define $\psi^*(t)=\psi(t)$ for
$t\geq\delta$ and let $\psi^*(t)=\psi_0(t)+a$ for $t\in(0,\delta)$
with a sufficiently large $a$, so that $\psi^*$ is monotone
right-continuous, and $\chi^{\mathrm{hc}}_{\psi^*}$ is a regularity modulus.
Applying the previous arguments to $\psi^*$ yields that there exists
a point process $\xi$ such that $\mathbf{E}\chi^{\mathrm{hc}}_{\psi
^*}(\xi) \leq r$.
Since $r<\infty$, $\rho$ vanishes on
$\{(x,y)\dvtx  \mathbf{d}(x,y)<\delta\}$, and so $\mathbf{E}\chi
^{\mathrm{hc}}_{\psi^*}(\xi)=\mathbf{E}
\chi^{\mathrm{hc}}_{\psi}(\xi) \leq r$.
\end{pf}

The following result is obtained by letting $\psi$ be infinite on
$[0,\varepsilon)$ and otherwise setting it to zero.

%
\begin{corollary}
\label{corhardcore}
A measure $\rho$ on $\mathbb{X}\times\mathbb{X}$ is the correlation
measure of a
point~process $\xi$ with $\xi\in\mathcal{N}_{\varepsilon}$ a.s. if
and only if
(\ref{eqtot-posit}) holds and $\rho(\{(x,y)\dvtx\break
\mathbf{d}(x,y)<\varepsilon\})=0$.
\end{corollary}

The following result yields a direct realisability condition for
$\rho$ without mentioning a regularity modulus.

%
\begin{theorem}
\label{thrreal-hardcore-pp}
Let $\rho$ be a locally finite measure on $\mathbb{X}\times\mathbb
{X}$, and fix any
\mbox{$r\geq0$}. Then there exists, for every $r'>r$, a simple point
process $\xi$ with correlation measure $\rho$, such that
%
\begin{equation}
\label{eqhardcore-xi} \mathbf{E}\sum_{x_i,x_j \in\xi, i\neq j}
P_{\mathbf{d}(x_i,x_j)}(\mathbb{X})\leq r',
\end{equation}
if and only if (\ref{eqtot-posit}) holds and
%
\begin{equation}
\label{condreal-rho} \int_{\mathbb{X}\times\mathbb{X}}P_{\mathbf
{d}(x,y)}(\mathbb{X})
\rho(dx\,dy)\leq r.
\end{equation}
\end{theorem}

\begin{pf}
\textit{Necessity}. Call $h_{t}(x,y)=\min(t,P_{\mathbf
{d}(x,y)}(\mathbb{X}))$ for
$x\neq y\in\mathbb{X}$ and $t>0$. Assume that $\xi$ realises $\rho$
and satisfies (\ref{eqhardcore-xi}). The monotone convergence
theorem yields that
\[
\int_{\mathbb{X}\times\mathbb{X}}P_{\mathbf{d}(x,y)}(\mathbb{X})\rho
(dx\,dy) =\lim
_{t\to\infty}\mathbf{E}\mathsf{g}_{h_{t}}(\xi) \leq
r'
\]
for every $r'>r$, whence (\ref{condreal-rho}) holds.

\textit{Sufficiency}.
Define a measure on $\mathbb{R}_+$ by
\[
\rho'\bigl([a,b)\bigr)=\rho\bigl(\bigl\{(x,y)\in\mathbb{X}\times
\mathbb{X}\dvtx  a\leq\mathbf{d}(x,y)<b\bigr\}\bigr).
\]
Fubini's theorem yields that
\[
r=\int_{\mathbb{R}_+}P_{t}(\mathbb{X})
\rho'(dt).
\]
Let $\{t_{k}, k\geq1\}$ be a strictly decreasing sequence of
numbers such that
\[
\int_{[0,t_{k})}P_{t}(\mathbb{X})\rho'(dt)
\leq2^{-k}.
\]
For $m\geq1$, the function
\[
\psi_m(t)= %
\cases{ kP_{t}(\mathbb{X}), &\quad
if $t_{k+1}\leq t < t_k< t_m, k\geq1$,
\vspace*{3pt}\cr
P_{t}(\mathbb{X}), &\quad if $t\geq t_m$} %
\]
is monotone right-continuous and satisfies $\psi_m(t)/P_{t}(\mathbb
{X})\to
\infty$ as $t\to0$. Then
\begin{eqnarray*}
\int_{\mathbb{X}\times\mathbb{X}}\psi_m\bigl(\mathbf{d}(x,y)\bigr)
\rho(dx\,dy) &=&\int_{\mathbb{R}_+}\psi_m(t)
\rho'(dt)
\\
&\leq&\int_{\mathbb{R}_+}P_{t}(\mathbb{X})
\rho'(dt)+\sum_{k\geq
m}k 2^{-k}
\leq r+\sum_{k\geq m}k 2^{-k}.
\end{eqnarray*}
By Theorem~\ref{thrreal-pp-psi}, choosing $m$ sufficiently large
yields the realisability of $\rho$ by a point process $\xi$
satisfying
\[
\mathbf{E}\sum_{x_i,x_j \in\xi, i\neq j} P_{\mathbf{d}(x_i,x_j)}(\mathbb{X})
\leq\mathbf{E}\chi^{\mathrm
{hc}}_{\psi_m}(\xi) \leq r+\sum
_{k\geq m}k 2^{-k}<r'.
\]\upqed
\end{pf}

%
\begin{remark}
\label{remstat-hc}
Let $\Theta$ be a group of continuous transformations on $\mathbb{X}$
that leave $\rho$ invariant, that is, $\rho(\theta A\times\theta
B)=\rho(A\times B)$ for all $\theta\in\Theta$ and Borel $A,B$. Since
the regularity modulus $\chi^{\mathrm{hc}}_\psi$ can be approximated
from below
by a sequence of functions from $\mathsf{G}$,
Theorem~\ref{thrchi-approx}(i) is applicable and so the corresponding
point process $\xi$ in
Theorems~\ref{thrreal-pp-psi},~\ref{thrreal-hardcore-pp} and
Corollary~\ref{corhardcore} can be chosen $\Theta$-stationary. If
$\Theta$ consists of isometric transformations, then
Theorem~\ref{thminvar-single}(ii) is also applicable.
\end{remark}

\subsection{Noncompact case and stationarity}\label{secnon-compact-case}

Assume that $\mathbb{X}=\mathbb{R}^d$ and $\mathbf{d}(x,y)=\|x-y\|$
is the Euclidean
metric. Let $\psi$ be a positive right-continuous monotone function on
$\mathbb{R}_+$ such that $\psi(t)t^d\to\infty$ as $t\to0$. Denote
by $B_n$
the open ball of radius $n$ centred at $0$. Given a known bound for
the packing number in the Euclidean space (\cite{mati95}, page~78,
Lemma~\ref{thrpsi-rm}) implies that $\chi^{\mathrm{hc}}_{\psi}$ is
a regularity
modulus on every $B_n$, $n\geq1$. Define
%
\begin{equation}
\label{eqregmod-non-comp} \chi^{\mathrm{hc}}_{\beta\psi}(Y)=\sum
_{x_i,x_j \in Y, i\neq j} \beta(x_i,x_j)\psi\bigl(
\|x_i-x_j\|\bigr)
\end{equation}
for a bounded lower semi-continuous strictly positive on
$\mathbb{R}^d\times\mathbb{R}^d$ function~$\beta$.

%
\begin{theorem}
\label{thrnon-comp}
Let $\rho$ be a locally finite measure on $\mathbb{R}^d\times\mathbb{R}^d$.
\begin{longlist}[(ii)]
\item[(i)] The measure $\rho$ is realisable as the
correlation measure of a point process $\xi$ that satisfies
$\mathbf{E}\chi^{\mathrm{hc}}_{\beta\psi}(\xi)\leq r$ if and only if
(\ref{eqtot-posit}) holds and
%
\begin{equation}
\label{eqintegr-condition} \int_{\mathbb{R}^d\times\mathbb{R}^d}\beta
(x,y)\psi\bigl(\|x-y\|\bigr)\rho(dx\,dy)
\leq r.
\end{equation}
\item[(ii)] Fix $r\geq0$, let
%
\begin{equation}
\label{eqrn=-bnx-y} r_n=\int_{B_n\times B_n}\|x-y
\|^{-d}\rho(dx\,dy), \qquad n\geq1,
\end{equation}
and let $\{\beta_n, n\geq1\}$ be a sequence of nonincreasing
numbers converging to $0$. Then the following assertions are
equivalent.
\begin{enumerate}[(a)]
\item[(a)] Equation (\ref{eqtot-posit}) holds and
%
\begin{equation}
\label{eqbetan} \sum_{n\geq1}\beta_n(r_{n+1}-r_n)
\leq r<\infty,
\end{equation}
in particular every $r_{n}$, $n\geq1$, is finite.
\item[(b)] For every $r'>r$, there exists $\xi$ with correlation
measure $\rho$ and such that
%
\begin{equation}
\label{eqbound-betan} \sum_{n\geq1}(\beta_n-
\beta_{n+1}) \mathbf{E}\sum_{x_i,x_j\in B_n, i\neq j}
\|x_i-x_j\|^{-d} \leq r'.
\end{equation}
\end{enumerate}
\end{longlist}
\end{theorem}

\begin{pf}
\textit{Sufficiency}.
\textup{(i)} The function $\chi^{\mathrm{hc}}_{\beta\psi}$ is a regularity
modulus on $\mathcal{N}_0$, since
\[
\mathcal{H}_{c,h}=\bigl\{Y\in\mathcal{X}\dvtx  \chi^{\mathrm{hc}}_{\beta
\psi}(Y)
\leq c+\mathsf{g}_h(Y)\bigr\},\qquad c\in\mathbb{R}, h\in\mathscr
{C}_{\mathrm{o}},
\]
is compact in $\mathcal{N}_0$. This follows from Lemma~\ref{thrpsi-rm},
which yields the compactness of the restriction of $Y$ from
$\mathcal{H}_{c,h}$ onto any compact set $C$. Indeed, this family of
restricted counting measures coincides with the family of simple
counting measures supported by $C$ such that $\chi^{\mathrm{hc}}_\psi
(Y)\leq
c/m+ \mathsf{g}_{h/m}(Y)$, where $m>0$ is a lower bound of $\beta
(x,y)$ for
$x,y\in C$.

In order to apply Theorem~\ref{thrreal-reg} with the regularity
modulus (\ref{eqregmod-non-comp}) and in view of (\ref{eqlip}) it
suffices to show that
%
\begin{equation}
\label{eqreg-pp-noncompact} \inf_{Y\in\mathcal{N}_{0}} \bigl[\chi
^{\mathrm{hc}}_{\beta\psi
}(Y)-
\mathsf{g}_h(Y) \bigr] +\int_{\mathbb{R}^d\times\mathbb{R}^d}h(x,y)
\rho(dx\,dy)\leq r
\end{equation}
for all $h\in\mathscr{C}_{\mathrm{o}}$. Assume that $h$ is supported
by a subset of
$B_n\times B_n$ for some $n\geq1$. Then
(\ref{eqreg-pp-noncompact}) holds by the same reasoning as in the
proof of Theorem~\ref{thrreal-pp-psi} applied to the compact space
$B_n$. (One might first consider only $Y\subset B_{n}$, and then note
that the infimum over all $Y\in\mathcal{N}_0$ is necessarily
smaller.) By
Theorem~\ref{thrreal-reg}, (\ref{eqintegr-condition}) implies the
existence of a point process $\xi$ with correlation measure $\rho$
that satisfies $\mathbf{E}\chi^{\mathrm{hc}}_{\beta\psi}(\xi)\leq r$.\vspace*{2pt}

\textup{(ii)} Define $\mathbb{Y}_n=(B_n\times B_n)\setminus
(B_{n-1}\times B_{n-1})$, $n\geq1$ (with $B_0=\varnothing$). For
every $n\geq1$, define the measure
\[
\rho'_n\bigl([a,b)\bigr)=\rho\bigl(\bigl\{(x,y)\in
B_n\times B_n\dvtx  a\leq\|x-y\| < b\bigr\}\bigr),
\]
and let
\[
t_k^n=\sup\biggl\{t>0\dvtx  \int_{[0,t)}s^{-d}
\rho'_n(ds)\leq2^{-k} \biggr\}, \qquad k\geq1.
\]
Since $\rho'_{n+1}\geq\rho'_n$ for every $n\geq1$, for every
$k,n\geq1$ we have $t_k^{n+1}\leq t_k^n$. Let $\{m_n, n\geq1\}$
be a nondecreasing sequence of positive integers so that
%
\begin{equation}
\label{eqsum-rhoprime} \sum_{n\geq1}\beta_n
\biggl(r_n-r_{n-1}+\sum_{k\geq m_{n}}k2^{-k}
\biggr)\leq r'.
\end{equation}
Now define
\[
\psi_n(t)= %
\cases{ k t^{-d}, &\quad if
$t_{k+1}^n \leq t <t_k^n<
t_{m_n}^n$,
\cr
t^{-d}, &\quad if $t\geq
t_{m_n}^n$.} %
\]
Since $m_n\leq m_{n+1}$, $\psi_{n+1}\leq\psi_n$ for every $n \geq
1$. Function $\psi_n$ satisfies $\psi_n(t)t^d\to\infty$ as $t\to
0$, whence, for every $n\geq1$, $\chi^{\mathrm{hc}}_{\psi_n}$ is a
regularity
modulus on counting measures supported by $B_n$ and
\[
\int_{\mathbb{Y}_n}\psi_n\bigl(\|x-y\|\bigr)\rho(dx\,dy) \leq\int
_{\mathbb{R}_{+}}\psi_n(t)\rho''_n(dt),
\]
where
\[
\rho''_n\bigl([a,b)\bigr)=\rho\bigl(\bigl
\{(x,y)\in\mathbb{Y}_n\dvtx  a\leq\|x-y\|<b\bigr\}\bigr) \leq
\rho_n'\bigl([a,b)\bigr).
\]
Then
\[
\int_{\mathbb{Y}_n}\psi_n\bigl(\|x-y\|\bigr)\rho(dx\,dy) \leq\int
_{\mathbb{R}_+}t^{-d}\rho''_n(dt)+
\int_{t\leq
t_{m_n}}\psi_n(t)\rho''_n(dt).
\]
We have
\[
\int_{\mathbb{R}_+}t^{-d}\rho''_n(dt)
=\int_{(B_n\times B_n)\setminus( B_{n-1}\times B_{n-1})} \|x-y\|
^{-d}\rho(dx\,dy)=r_n-r_{n-1}
\]
and
\[
\int_{t\leq t_{m_n}} \psi_n(t)\rho''_n(dt)
\leq\int_{t\leq t_{m_n}} \psi_n(t)\rho_n'(dt)
\leq\sum_{k\geq m_n}k2^{-k},
\]
whence
%
\begin{equation}
\label{eqsum-rn-rn-1} \int_{\mathbb{Y}_n}\psi_n\bigl(\|x-y\|\bigr)
\rho(dx\,dy) \leq r_n-r_{n-1}+\sum
_{k\geq m_n}k2^{-k}.
\end{equation}
Define $\psi(x,y)=\psi_n(\|x-y\|)$ for $x,y\in\mathbb{Y}_n$. Since
$\psi_{n+1}\leq\psi_n$ and functions $\psi_n$, $n\geq1$, are lower
semi-continuous, the function $\psi$ is lower semi-continuous on
$\mathbb{R}^d\times\mathbb{R}^d$. Define $\beta(x,y)=\beta_n$ on
$\mathbb{Y}_n$. Since
$\beta_n$, $n\geq1$, decrease, $\beta$ is a lower semi-continuous
function on $\mathbb{R}^d\times\mathbb{R}^d$. Since $\psi_n \leq
\psi_k$ for every
$k\leq n$, the restriction of $\chi^{\mathrm{hc}}_{\beta\psi}$ onto\vspace*{2pt} sets
$Y\subset B_n$ is larger than $\chi^{\mathrm{hc}}_{\beta_n\psi_n}$, whence
$\chi^{\mathrm{hc}}_{\beta\psi}$ is a regularity modulus on
$\mathcal{N}_0$. By
Theorem~\ref{thrreal-reg}, $\Phi$ is realised by a point process
$\xi$ satisfying
\[
\mathbf{E}\chi^{\mathrm{hc}}_{\beta\psi}(\xi) \leq\int_{\mathbb
{R}^d\times\mathbb{R}^d}
\beta(x,y)\psi(x,y)\rho(dx\,dy) =\sum_{n\geq1}
\beta_n\int_{\mathbb{Y}_n}\psi_n\bigl(\|x-y\|\bigr)\rho
(dx\,dy).
\]
Since $t^{-d}\leq\psi_n(t)$ for each $n$ and $t>0$,
\begin{eqnarray*}
\mathbf{E}\chi^{\mathrm{hc}}_{\beta\psi}(\xi) &=&\lim_{m\to\infty}
\mathbf{E}\sum_{n=1}^m\beta_n
\Biggl(\chi^{\mathrm
{hc}}_{\psi_n}(\xi\cap B_n) -
\chi^{\mathrm{hc}}_{\psi_n}(\xi\cap B_{n-1})
\\
&&\hspace*{69pt}
\geq\lim_{m\to\infty} \mathbf{E}\sum_{n=1}^m
(\beta_n-\beta_{n+1})\chi^{\mathrm{hc}}_{ \psi_n}(\xi
\cap B_n)\Biggr)
\\
&\geq&\sum_{n\geq1}(\beta_n-
\beta_{n+1}) \mathbf{E}\sum_{i\neq j,x_i,x_j\in B_n}
\|x_i-x_j\|^{-d}.
\end{eqnarray*}
Using successively (\ref{eqsum-rn-rn-1}) and
(\ref{eqsum-rhoprime}),
\[
\sum_{n\geq1}\beta_n\int
_{\mathbb{Y}_n}\psi_n\bigl(\|x-y\|\bigr)\rho(dx\,dy) \leq\sum
_{n\geq
1}\beta_n\biggl(r_n-r_{n-1}+
\sum_{k\geq m_n}k2^{-k}\biggr) \leq
r',
\]
we arrive at (\ref{eqbound-betan}).

\textit{Necessity}. For \textup{(ii),} remark first that for $x\neq
y\in\mathbb{R}^{d}$
\[
\beta(x,y)\|x-y\|^{-d}= \sum_{n\geq1}(
\beta_n-\beta_{n+1}) \mathbh{1}_{x,y\in B_n}\|x-y
\|^{-d},
\]
where $\beta$ is defined in the sufficiency part of the proof. The
function $\beta\psi$ in \textup{(i)} and the function $(x,y)\mapsto
\beta(x,y)\|x-y\|^{-d}$ in \textup{(ii)} are lower semi-continuous
and can therefore be approximated from below by compactly supported
continuous functions. The rest follows from the monotone convergence
theorem similarly to the proof of necessity in
Theorem~\ref{thrreal-hardcore-pp}.
\end{pf}

%
\begin{remark}\label{rkbetan-exist}
Remark for point \textup{(ii)} that if each $r_{n}$, $n\geq1$, is
finite, there always exists a sequence $\{\beta_{n}\}$ of
sufficiently small numbers such that the right-hand side of
(\ref{eqbetan}) is finite.
\end{remark}

If the distribution of point process $\xi$ is invariant with respect
to the group $\Theta$ of translations of $\mathbb{R}^d$, then $\xi$
is called
\emph{stationary}. Its correlation measure $\rho$ is translation
invariant, that is, $\rho((A+x)\times(B+x))=\rho(A\times B)$ for all
$x\in\mathbb{R}^d$ and so
%
\begin{equation}
\label{eqrhoat-b=rh-bbarrh} \rho(A\times B)=\lambda^2 \int_A
\int_{\mathbb{R}^d}\mathbh{1}_{x+y\in B}\bar{\rho}(dy)\,dx,
\end{equation}
where $\lambda$ is the intensity of $\xi$ and $\bar\rho$ is a measure
on $\mathbb{R}^d$ called the \emph{reduced} correlation measure; see
\cite{schweil08}, page~76.

%
\begin{theorem}
\label{thstat}
Let $\bar\rho$ be a locally finite measure on $\mathbb{R}^d$, let
$\beta$ be
a bounded lower semi-continuous strictly positive function on $\mathbb{R}^d$
satisfying
\[
\bar\beta(y)=\int_{\mathbb{R}^d}\beta(x,x+y)\,dx<\infty,\qquad y\in
\mathbb{R}^d,
\]
and let $\psi$ be a monotone decreasing nonnegative function such
that $t^d\psi(t)\to\infty$.
\begin{longlist}[(ii)]
\item[(i)] $ \bar\rho$ is the reduced correlation measure
of a stationary point process $\xi$ that satisfies
$\mathbf{E}\chi^{\mathrm{hc}}_{\beta\psi}(\xi)\leq r$ if and only if
(\ref{eqtot-posit}) holds and
\[
\int_{\mathbb{R}^d}\bar\beta(y)\psi\bigl(\|y\|\bigr)\bar\rho(dy)\leq r.
\]
\item[(ii)] $\bar\rho$ is realisable as the reduced
correlation measure of a stationary point process $\xi$ that
satisfies (\ref{eqbound-betan}) for some sequence
$\{\beta_{n}, n\geq1\}$ if and only if
%
\begin{equation}
\label{eqrhobar-regular} \int_{B}\|y\|^{-d}\bar\rho(dy)<
\infty
\end{equation}
for some open ball $B$ containing the origin. If
$\int_{\mathbb{R}^d}\|y\|^{-d}\bar\rho(dy)$ is finite, it is
possible to
let $\beta_n=n^{-d-\delta}$, $n\geq1$, for any fixed $\delta>0$.
\end{longlist}
\end{theorem}

\begin{pf}
It suffices to use (\ref{eqrhoat-b=rh-bbarrh}) to confirm the
conditions imposed in Theorem~\ref{thrnon-comp}, see also
Remark~\ref{rkbetan-exist}. In order to show that $\xi$ can be
chosen stationary, note that $\chi^{\mathrm{hc}}_{\beta\psi}$ can
be pointwisely
approximated from below by a monotone sequence of functions from
$\mathsf{G}$, so Theorem~\ref{thrchi-approx}(i) applies.
\end{pf}

\subsection{Joint realisability of the intensity and correlation}

Recall that the intensity measure $\rho_1$ of a point process $\xi$ is
defined from
\[
\mathbf{E}\sum_{x_i\in\xi} h(x_i)=\int h(x)
\rho_1(dx),\qquad h\in\mathscr{C}_{\mathrm{o},1},
\]
where $\mathscr{C}_{\mathrm{o},1}$ is the family of continuous
functions on $\mathbb{X}$
with compact support. A~pair $(\rho_1,\rho)$ of locally finite
nonnegative measures on $\mathbb{X}$ and $\mathbb{X}\times\mathbb
{X}$, respectively, is
said to be jointly realisable if there exists a point process $\xi$
with intensity measure $\rho_1$ and correlation measure $\rho$.

Let $\mathsf{G}_1$ be the vector space formed by constants and functions
\[
\mathsf{g}_{h_1,h}(Y)=\sum_{x\in Y}h(x)+
\mathsf{g}_h(Y), \qquad Y\in\mathcal{N},
\]
for $h_1\in\mathscr{C}_{\mathrm{o},1}$ and $h\in\mathscr
{C}_{\mathrm{o}}$. It is easy to see that
Assumption~\ref{assumpg} is verified in this case. The pair
$(\rho_1,\rho)$ yields a linear functional
%
\begin{equation}
\label{eqfc+gh1-h=c+-+ints} \Phi(\mathsf{g}_{h_1,h})=\int_{\mathbb{X}}h_1(x)
\rho_{1}(dx) +\int_{\mathbb{X}\times\mathbb{X}}h(x,y)\rho(dx\,dy).
\end{equation}
The realisability of $\Phi$ by a point process $\xi$ means that
$\Phi(\mathsf{g}_{h_1,h})=\mathbf{E}\mathsf{g}_{h_1,h}(\xi)$.
Functional $\Phi$ is positive on
$\mathsf{G}_1$ if and only if
%
\begin{equation}
\label{eqfgh1-hgeq-infy} \Phi(\mathsf{g}_{h_1,h})\geq\inf_{Y\in\mathcal{X}}
\mathsf{g}_{h_1,h}(Y), \qquad h_1\in\mathscr{C}_{\mathrm{o},1},
h\in\mathscr{C}_{\mathrm{o}}.
\end{equation}

Similar arguments as before apply and yield the joint realisability
conditions. Consider the special case of stationary processes in
$\mathbb{X}=\mathbb{R}^d$ with the reduced correlation measure $\bar
\rho$ [see
(\ref{eqrhoat-b=rh-bbarrh})] and intensity $\rho_1(dx)=\lambda \,dx$
proportional to the Lebesgue measure.

%
\begin{theorem}
Let $\lambda$ be a constant, and let $\bar\rho$ be a locally finite
measure of~$\mathbb{R}^d$. Then there is a stationary point process
$\xi$
with intensity $\rho_{1}(dx)=\lambda \,dx$ and reduced correlation
measure $\bar\rho$ if $\Phi$ given by
(\ref{eqfc+gh1-h=c+-+ints}) satisfies
(\ref{eqfgh1-hgeq-infy}) with $\mathcal{X}=\mathcal{N}_0$ and
\[
\int_{B}\|z\|^{-d}\bar\rho(dz)<\infty
\]
for some open set $B$ containing the origin.
\end{theorem}

\begin{pf}
It suffices to note that $\mathsf{g}_{h_1,h}$ is dominated by
$c\mathsf{g}_h$ for
a constant $c$ and follow the proof of \textup{(ii)} in
Theorem~\ref{thrnon-comp}. The condition on $\bar\rho$ follows
from (\ref{eqrn=-bnx-y}) and (\ref{eqrhoat-b=rh-bbarrh}).
\end{pf}

\section{Realisability of covering probabilities for random sets}\label{secreals-under-smoothn}

The nature of realisability problem changes with the choice of the
family of subsets of a carrier space $\mathbb{X}$ taken as possible
values for a random set. We start with the case when a random set is
allowed to be any subset of $\mathbb{X}$, where realisability results
are available under minimal conditions, while the obtained random set
lacks properties and might even not be measurable. In the remainder of
this section, we treat the case of random closed sets, a classical
setting in stochastic geometry. Some examples of possible regularity
moduli are presented, along with the corresponding realisability
results that resemble those of \cite{kunlebspeer07} in the point
processes setting. The framework of random measurable sets with finite
perimeter (in the variational sense), developed in the forthcoming
paper \cite{GalLac}, provides a compromise between regularity of the
random set and the explicitness of conditions.

\subsection{Random binary functions}\label{secindic-rand-funct}

Let $\mathcal{X}$ be the family of all subsets of $\mathbb{X}$
identified with
their indicator functions. Endow $\mathcal{X}$ with the topology of pointwise
convergence and the corresponding $\sigma$-algebra.
Since $\mathcal{X}$ is compact, Corollary~\ref{correal-reg}
yields the following result.

%
\begin{theorem}
\label{thrindic}
Let $\mathsf{G}$ be a vector space that consists of continuous
functions on $\mathcal{X}$ and includes constants, and let $\Phi$ be
a map
from $\mathsf{G}$ to $\mathbb{R}$. Then there exists a random indicator
function $\xi$, such that $\Phi(\mathsf{g})=\mathbf{E}\mathsf
{g}(\xi)$ for all $\mathsf{g}\in
\mathsf{G}$ if and only if $\Phi$ is a linear positive functional on
$\mathsf{G}$ and $\Phi(1)=1$.
\end{theorem}

The key issue in applying Theorem~\ref{thrindic} is the choice of the
space $\mathsf{G}$.

%
\begin{example}[(One-point covering function)]\label{exopc}
Let $\mathsf{G}$ be generated by constants $c$ and one-point indicator
functions $\mathsf{g}_x(F)=\mathbh{1}_{x\in F}$, $F\in\mathcal{X}$,
for $x\in\mathbb{X}$.
The positivity of a linear functional $\Phi\dvtx \mathsf{G}\mapsto\mathbb{R}$
together with $\Phi(1)=1$ means that $\Phi(\mathsf{g}_x)\in[0,1]$
for all
$x\in\mathbb{X}$. Thus, a function $p_x=\Phi(\mathsf{g}_x)$ is a one-point
covering function $\mathbf{P}\{x\in\xi\}$ for a random set $\xi$ if and
only if $p_x$ takes values in $[0,1]$. Compare with
Theorem~\ref{throp}, where the extra upper semi-continuity
condition ensures that the corresponding random binary function is
upper semi-continuous and so $\xi$ is a random \emph{closed} set.
\end{example}

%
\begin{example}[(Covariances of random sets)]\label{excrs}
Consider vector space $\mathsf{G}$ generated by constants and functions
$\mathsf{g}_{x,y}(F)=\mathbh{1}_{x,y\in F}$ for $x,y\in\mathbb{X}$.
The values of
a linear functional $\Phi$ on $\mathsf{G}$ are determined by
$p_{x,y}=\Phi(\mathsf{g}_{x,y})$, $x,y\in\mathbb{X}$.
By~(\ref{eqphi-positivity}), $\Phi$ is positive on $\mathsf{G}$ if and
only if
%
\begin{equation}
\label{eqpos-cov-cond} \sum_{ij=1}^n
a_{ij} p_{x_i,x_j} \geq\inf_{F\subset\mathbb{X}} \sum
_{ij=1}^n a_{ij} \mathbh
{1}_{x_i,x_j\in F}
\end{equation}
for\vspace*{1.5pt} all $n\geq1$ and all matrices $(a_{ij})_{ij=1}^n$. In
particular, if $a_{ij}=a_ia_j$, then (\ref{eqpos-cov-cond}) implies
the nonnegative definiteness of $p_{x,y}$, $x,y\in\mathbb{X}$.
Note that the one-point covering probabilities are specified if
$p_{x,y}$ are given.
\end{example}

\subsection{The closedness condition}\label{secreal-rand-clos}

A \emph{random closed set} $\xi$ in a locally compact metric space
$\mathbb{X}$ is a random element that takes values in the family
$\mathcal{X}=\mathcal{F}$ of closed subsets of $\mathbb{X}$ equipped
with the
$\sigma$-algebra (called the Effros $\sigma$-algebra) generated by
families $\{F\in\mathcal{F}\dvtx  F\cap K\neq\varnothing\}$ for all
compact sets
$K$. The distribution of a random closed set $\xi$ is uniquely
determined by its \emph{capacity functional}
\[
T(K)=\mathbf{P}\{\xi\cap K\neq\varnothing\}
\]
for all $K$ from the family of all compact sets in $\mathbb{X}$; see
\cite{ma} and \cite{mo1}, Theorem~1.1.13.

Theorem~\ref{thrindic} ensures only the existence of a binary
stochastic process with given marginal distributions up to a certain
order. However, it is not guaranteed that the constructed stochastic
process has upper semi-continuous realisations, which should be the
case if this process is the indicator of a random \emph{closed} set in
a topological space $\mathbb{X}$. If the carrier space $\mathbb{X}$ is
finite (more generally, has a discrete topology), then this problem is
avoided, since each random set is closed. Furthermore, the closedness
issue can be settled in the following special case of two-point
probabilities in the product form (and can be generalised for
multi-point covering probabilities). The following result implies, in
particular, that the random indicator function from
Example~\ref{exproduct} does not correspond to a random closed
set. It also illustrates regularity problems arising in realisability
problems for random closed sets.

%
\begin{theorem}
\label{propproduct-form}
Assume that $\mathbb{X}$ is a separable space. A function
\[
p_{x,y}= %
\cases{ p_{x}p_{y}, &\quad
if $x\neq y$,
\vspace*{2pt}\cr
p_{x}, &\quad if $x=y$} %
\]
is a two-point covering function of a random closed set if and only
if $p_x$, $x\in\mathbb{X}$, is an upper semi-continuous function with values
in $[0,1]$ such that each point $x$ with $p_x\in(0,1)$ has an open
neighbourhood $U$ such that $p_y>0$ only for at most a countable
number of $y\in U$ and the sum of $p_y$ for $y\in U$ is finite.
\end{theorem}

\begin{pf}
\textit{Sufficiency}. Note that the set $L=\{x\dvtx  p_x=1\}$ is
closed by the upper semi-continuity of $p_x$. The separability of
$\mathbb{X}$ and the condition of theorem imply that the set $M=\{x\dvtx
p_x\in(0,1)\}$ is at most countable. The sufficiency is obtained
by a direct construction of a random subset $Z$ of $M$ that contains
each point $x$ with probability $p_x$ independently of all other
points. It remains to show that the random set $\xi=Z\cup L$ is
closed. Consider $x\in M$ and its neighbourhood from the condition
of theorem. Since $\sum p_y<\infty$, only a finite number of $y$
belong to $Z$ and so they do not converge to $x$. Thus, $x$ with
probability zero appears as a limit of other points from $\xi$ unless
$x\in L$ and so belongs to $\xi$ almost surely.

\textit{Necessity}. The function $p_x=\mathbf{P}\{x\in\xi\}$ is upper
semi-continuous, since $\xi$ is a random closed set. The product
form of
the two-point covering function implies that the capacity functional
on two-point set is given by
\[
T\bigl(\{x,y\}\bigr)=p_x+p_y-p_xp_y.
\]
The upper semi-continuity property of the capacity functional yields
that
\[
\limsup_{y\to x} T\bigl(\{x,y\}\bigr)\leq p_x,
\]
while the monotonicity implies that $T(\{x,y\})\to p_x$ as $y\to x$.
Thus, $p_y(1-p_x)\to0$ as $y\to x$ for all $x$. Unless $p_x=1$, we
have $p_y\to0$.

Assume that $p_x>0$ and $p_{x_n}>0$, where $x_n\to x$ and $x_n\neq
x$ with $\sum p_{x_n}=\infty$. A variant of the lemma of
Borel--Cantelli for pairwise independent events (see~\cite{frisgray97}, Lemma~6.2.5) implies that almost surely
infinitely many points $x_n$ belong to $\xi$, so that $x\in\xi$
a.s. by the closedness of $\xi$ and so $p_x=1$. Thus, the sum of
$p_{x_n}$ for each sequence $\{x_n\}$ in a neighbourhood of $x$ is
finite. This rules out the existence of uncountably many $y$ with
$p_y>0$ in any neighbourhood of $x$. Indeed, then $\{y\dvtx  p_y\geq
1/n\}$ is finite for all $n$, and so the union of such sets is
countable.
\end{pf}

It is known that $\mathcal{F}$ is compact and completely regular in
the Fell
topology that generates the Effros $\sigma$-algebra; see
\cite{kur66}, Theorem~17.V.3 and \cite{mo1}. However,
Corollary~\ref{correal-reg} is not applicable, since functions
$\mathbh{1}_{x,y\in F}$, $F\in\mathcal{F}$, generating the vector
space $\mathsf{G}$,
do not generate the Effros $\sigma$-algebra on $\mathcal{F}$.

It is known (\cite{mo1}, Theorem~1.2.6) that the $\sigma$-algebra generated
by $\mathsf{G}$ on the family of \emph{regular closed} sets (that
coincide with closures of their interiors) coincides with the trace of
the Effros $\sigma$-algebra on the family of regular closed sets.
However, the family of regular closed sets is no longer compact in the
Fell topology. Furthermore, indicator functions are not continuous in
the Fell topology, so it is again not possible to appeal to
Corollary~\ref{correal-reg} or explicitly check the upper
semi-continuity condition required in Daniell's theorem.

In order to describe a useful family $\mathsf{G}$ of functionals acting
on sets, consider a $\sigma$-finite measure $\nu$ on
$\mathbb{X}$ and define
%
\begin{equation}
\label{eqgh-set} \mathsf{g}_h(F) =\int_{F\times F}h(x,y)
\nu(dx)\nu(dy)
\end{equation}
for all measurable $F\subset\mathbb{X}$ and $h$ from the family
$\mathscr{C}_{\mathrm{o}}$ of
symmetric continuous functions with compact support in
$\mathbb{X}\times\mathbb{X}$. A function $p_{x,y}$, $x,y\in\mathbb
{X}$, generates a
functional acting on $\mathsf{g}_h$ as
%
\begin{equation}
\label{eqfgh-set} \Phi(\mathsf{g}_h) = \int_{\mathbb{X}\times\mathbb{X}}
p_{x,y}h(x,y)\nu(dx)\nu(dy).
\end{equation}
The function $p_{x,y}$ is said to be \emph{weakly realisable} as the
two-point covering function if there exists a random set $\xi$
(or the corresponding random indicator function) such that $\xi$ is
almost surely measurable and $\mathbf{E}\mathsf{g}_h(\xi)=\Phi
(\mathsf{g}_h)$ for all
$h\in\mathscr{C}_{\mathrm{o}}$. By approximating the atomic masses
at two points with
continuous functions, it is easy to see that the weak realisability is
equivalent to $\Phi(\mathsf{g}_{x,y})=p_{x,y}$ for $\nu\otimes\nu
$-almost all
$(x,y)$, in contrast to the \emph{strong realisability} requiring this
equality everywhere.
The strong and weak realisability do not coincide in general. For
instance, a nonpositive function which vanishes almost everywhere,
but takes some negative values is weakly realisable by the empty set,
but not strongly realisable. Nevertheless, in the case of a
\emph{stationary} random regular closed set $\xi$ in $\mathbb{R}^d$
and the
Lebesgue measure $\nu$, the strong and weak realisability properties
coincide; see Theorem~\ref{thrreal-eps}.

In view of the required continuity property of functions from
$\mathsf{G}$, it is essential to ensure that $\mathsf{g}_h(F)$, $F\in
\mathcal{F}$,
defined in (\ref{eqgh-set}) is continuous in the Fell topology. Note
that it is not the case for most nontrivial measures $\nu$, even for
the Lebesgue measure. The continuity holds only on some subfamilies of
$\mathcal{F}$ considered in the following sections.

\subsection{Neighbourhoods of closed sets}\label{seceps}

For simplicity, in the following consider random sets in the Euclidean
space, that is, assume that $\mathbb{X}=\mathbb{R}^d$ with Euclidean metric
$\mathbf{d}$.

Let $\mathcal{F}^\varepsilon$ be the family of $\varepsilon
$-neighbourhoods of closed sets
in $\mathbb{R}^d$, that is, $\mathcal{F}^\varepsilon$ consists of
$F^\varepsilon=\{x\dvtx
\mathbf{d}(x,F)\leq\varepsilon\}$ for all $F\in\mathcal{F}$ and
also contains the empty
set. The vector space $\mathsf{G}$ is generated by constants and the
functions $\mathsf{g}_h$ defined by (\ref{eqgh-set}) with the Lebesgue
measure $\nu$.

%
\begin{lemma}
\label{lemmacont-eps}
The space $\mathcal{F}^\varepsilon$ with the Fell topology is compact
and, for each
$h\in\mathscr{C}_{\mathrm{o}}$, the function $\mathsf{g}_h$ is
continuous on $\mathcal{F}^\varepsilon$.
\end{lemma}

\begin{pf}
Recall that the upper limit of a sequence of sets $\{F_n\}$ is the
set of all limits for sequences $\{x_{n_k}\}$ such that $x_{n_k}\in
F_{n_k}$ for all $k$, while the lower limit is the set of all limits
for convergent sequences $\{x_n\}$ such that $x_n\in F_n$ for all
$n$. The sequence of closed sets converges in the Fell topology if
its upper and lower limits coincide.

If $F_n=F_{n,0}^\varepsilon\in\mathcal{F}^\varepsilon$ converges to
$F$ in the Fell
topology, then we can assume without loss of generality (by passing
to subsequences) that $F_{n,0}$ converges to $F_0$, so that
$F=F_0^\varepsilon$ and $F\in\mathcal{F}^\varepsilon$. Thus,
$\mathcal{F}^\varepsilon$ is a closed subset
of $\mathcal{F}$ and so is compact, since $\mathcal{F}$ is compact itself.

Consider a nonnegative $h\in\mathscr{C}_{\mathrm{o}}$ supported by
a ball $B_R$ centred
at the origin with sufficiently large radius $R$. If $F_n\to F$ in
the Fell topology, then the upper limit of $(F_n\cap B_R)$ is a
subset of $(F\cap B_R)$. Thus, $\mathsf{g}_h(F)=\mathsf{g}_h(F_n\cap
B_R)\leq
\mathsf{g}_h((F\cap B_R)^\delta)$ for any $\delta>0$ and
sufficiently large
$n$, so that $\mathsf{g}_h$ is upper semi-continuous on $\mathcal{F}$
and so on
$\mathcal{F}^\varepsilon$.

In order to prove the lower semi-continuity of $\mathsf{g}_h$ on
$\mathcal{F}^\varepsilon$
assume $F_n=F_{n,0}^\varepsilon\to F=F_0^\varepsilon$ and $F_{n,0}\to
F_0$. Fix
$\delta>0$. Then the lower limit of $F_{n,0}\cap B_{r+\delta}$ includes
$F_0\cap B_R$. Indeed, if $x\in F_0\cap B_R$, then $x_n\to x$ for
$x_n\in F_{n,0}$ and so $x_n\in B_{R+\delta}$ for all sufficiently large
$n$. Thus, for sufficiently large $n$, we have
\[
(F_0\cap B_R)\subset(F_{n,0}\cap
B_{R+\delta})^\delta.
\]
Taking $(\varepsilon-\delta)$-neighbourhoods of the both sides yields that
\[
(F_0\cap B_{R})^{\varepsilon-\delta}\subset(F_{n,0}
\cap B_{R+\delta})^{\varepsilon}\subset(F_{n}\cap
B_{R+\delta+\varepsilon
}).
\]
If $x\in F_{0}^{\varepsilon-\delta}\cap B_{R-\varepsilon+\delta}$,
then there is a
point $y\in F_{0}$ with $\mathbf{d}(x,y)\leq\varepsilon-\delta$, in
particular
$y\in B_{R}$. Thus,
\[
\bigl(F_{0}^{\varepsilon-\delta}\cap B_{R-\varepsilon+\delta}\bigr)
\subset(F_{n,0}\cap B_{R})^{\varepsilon-\delta}.
\]
Taking $r$ sufficiently large yields that $\mathsf{g}_h(F_n)\geq
\mathsf{g}_h(F_0^{\varepsilon-\delta})$. Since the interior of $F$ equals
$\bigcup_{\delta>0} F_0^{\varepsilon-\delta}$, the Lebesgue theorem yields
that $\mathsf{g}_{h}(F_0^{\varepsilon-\delta})\to\mathsf
{g}_{h}(F)$ as $\delta\to0$, that
is, $\mathsf{g}_h$ is lower semi-continuous on $\mathcal{F}^\varepsilon$.

For a nonpositive function $h$ with compact support, the result
follows by applying the above argument to its positive and negative
parts.
\end{pf}

%
\begin{theorem}
A function $p_{x,y}$, $x,y\in\mathbb{R}^d$, is weakly realisable by a random
closed set $\xi$ with realisations in $\mathcal{F}^\varepsilon$ for given
$\varepsilon>0$ if and only if
\[
\Phi(\mathsf{g}_h)\geq\inf_{F\in\mathcal{F}^\varepsilon} \mathsf
{g}_h(F),\qquad h\in\mathscr{C}_{\mathrm{o}},
\]
where $\Phi(\mathsf{g}_h)$ is given by (\ref{eqfgh-set}).
\end{theorem}

\begin{pf}
In view of the continuity of $\mathsf{g}_h$ established in
Lemma~\ref{lemmacont-eps}, it suffices to refer to
Corollary~\ref{correal-reg}.
\end{pf}

In order to handle random sets with realisations from the space
$\mathcal{F}^0=\bigcup_{\varepsilon>0} \mathcal{F}^\varepsilon$, we
need the regularity modulus
$\chi(F)$ defined as the infimum of $\varepsilon>0$ such that
$F\in\mathcal{F}^{1/\varepsilon}$ and $\chi(F)=\infty$ if $F\notin
\mathcal{F}^0$.

%
\begin{theorem}
\label{thrreal-eps}
For any given $r>0$, a function $p_{x,y}$, $x,y\in\mathbb{R}^d$, is weakly
realisable by a random closed set $\xi$ such that
$\mathbf{E}\chi(\xi)\leq r$ if and only if
%
\begin{equation}
\label{eqfepso} \inf_{F\in\mathcal{F}^0 }\bigl[ \chi(F)-\mathsf{g}_h(F)
\bigr] +\Phi(\mathsf{g}_h)\leq r,\qquad h\in\mathscr{C}_{\mathrm{o}},
\end{equation}
where $\Phi(\mathsf{g}_h)$ is given by (\ref{eqfgh-set}). If,
additionally, $p_{x,y}$ is an even continuous function of $x-y$,
then $p_{x,y}$ is strongly realisable by a stationary random closed
set $\xi$.
\end{theorem}

\begin{pf}
Function $\chi$ is lower semi-continuous, since $\{F\in\mathcal{F}\dvtx
\chi(F)\leq c\}=\mathcal{F}^{1/c}$ is closed for all $c>0$. Furthermore,
%
\begin{equation}
\label{eqfinsf-chifleq-ghf} \bigl\{F\in\mathcal{F}\dvtx  \chi(F)\leq\mathsf{g}_h(F)
\bigr\} \subset\bigl\{F\in\mathcal{F}\dvtx  \chi(F)\leq c\bigr\},
\end{equation}
where $c=\int|h(x,y)|\nu(dx)\nu(dy)$ is a finite upper bound for
$\mathsf{g}_h(F)$. The left-hand side of (\ref
{eqfinsf-chifleq-ghf}) is
compact, since $\mathsf{g}_h$ is continuous on $\mathcal{F}^{1/c}$ by
Lemma~\ref{lemmacont-eps} and the right-hand side of
(\ref{eqfinsf-chifleq-ghf}) is compact. Thus, $\chi$ is a
regularity modulus and the result follows from
Theorem~\ref{thrreal-reg} and (\ref{eqpositivity-with-constants}).

Note that the regularity modulus $\chi$ is invariant for the group
$\Theta$ of translations of~$\mathbb{R}^d$. By
Theorem~\ref{thminvar-single}(ii), $\xi$ can be chosen to be stationary.
In order to confirm the strong realisability, it remains to show
that the covariance function of a stationary regular closed random
set is continuous.

Since $\chi(\xi)$ is integrable, $\xi\in\mathcal{F}^0$, so that
$\xi$
is almost surely regular closed and its boundary $\partial\xi$
has a.s. vanishing Lebesgue measure. By Fubini's theorem, almost
every point $x$ belongs to the boundary of $\xi$ with probability
zero, and so $\mathbf{P}\{x\in\partial\xi\}=0$ for all $x$ in view
of the
stationarity property.

Let $\mathbf{P}\{x,y\in\xi\}$ be the covariance function of $\xi$. Take
$x,y\in\mathbb{R}^d$, and $(x_n,y_n)$ that converges to $(x,y)$.
Since with
probability $1$, $x$ does not belong to $\partial\xi$,
$\mathbh{1}_{x\in\xi}$ is almost surely equal to $\mathbh{1}_{x\in
\operatorname{Int}(\xi)}$ for the interior $\operatorname{Int}(\xi)$ of
$\xi$ and so $\mathbh{1}_{x_n\in\xi}$ almost surely converges to
$\mathbh{1}_{x\in\xi}$. Similarly, $\mathbh{1}_{y_n\in\xi}\to
\mathbh{1}_{y\in
\xi}$ a.s., whence the product converges too $\mathbh{1}_{x_n\in
\xi,y_n\in\xi}\to\mathbh{1}_{x\in\xi,y\in\xi}$. The
Lebesgue theorem yields that $\mathbf{P}\{x_n,y_n\in\xi\}\to\mathbf
{P}\{x,y\in
\xi\}$. Since $p_{x,y}$ and $\mathbf{P}\{x,y\in\xi\}$ are both
continuous and coincide almost surely, they are equal
everywhere. The continuity of $\mathbf{P}\{x,y\in\xi\}$ can be also
obtained by referring to a result of \cite{mo88} saying that the
capacity functional of each stationary regular closed random set is
continuous in the Hausdorff metric.
\end{pf}

\subsection{Convexity restrictions}\label{secsmoothm-restr}

The family $\mathcal{C}$ of convex closed sets in $\mathbb{R}^d$
(including the empty
set) is closed in the Fell topology and it is easy to see that the
function $\mathsf{g}_h$ given by (\ref{eqgh-set}) is continuous on
$\mathcal{C}$.
Corollary~\ref{correal-reg} yields that $p_{x,y}$ is weakly
realisable for a convex random closed set if and only if
\[
\Phi(\mathsf{g}_h)\geq\inf_{F\in\mathcal{C}}
\mathsf{g}_h(F)
\]
for the functional $\Phi(\mathsf{g}_h)$ given by (\ref{eqfgh-set}).

Let $\mathcal{P}$ be the \emph{convex ring} in $\mathbb{R}^d$, that
is, the family
of finite unions of compact convex subsets of $\mathbb{R}^d$. For
$F\in\mathcal{P}$,
let $\chi(F)$ be the smallest number $k$, such that $F$ can be
represented as the union of $k$ convex compact sets.

%
\begin{theorem}
Let $\Phi$ be linear functional defined by (\ref{eqfgh-set}). Fix
any $r>0$. Then there is a random closed set $\xi$ with
realisations in $\mathcal{P}$ such that $\mathbf{E}\mathsf{g}_h(\xi
)=\Phi(\mathsf{g}_h)$ for
all $h\in\mathscr{C}_{\mathrm{o}}$ and $\mathbf{E}\chi(\xi)\leq
r$ if and only if
\[
\inf_{F\in\mathcal{P}} \bigl[\chi(F)-\mathsf{g}_h(F) \bigr]+
\Phi(\mathsf{g}_h)\leq r,\qquad h\in\mathscr{C}_{\mathrm{o}}.
\]
\end{theorem}

\begin{pf}
The family $\mathcal{P}_k$ of unions of at most $k$ convex compact
sets is
closed in $\mathcal{F}$ and so is compact, whence $\chi$ is lower
semi-continuous. It is easily seen that $\mathsf{g}_h$ is continuous on
convex compact sets, and so is also continuous on $\mathcal{P}_k$. Thus,
$\mathsf{g}_h$ is $\chi$-regular and Theorem~\ref{thrreal-reg} applies.
\end{pf}

If $\mathbb{X}=[0,1]$, then $\mathcal{P}$ is be the family of finite
unions of
segments in $[0,1]$. The number of convex components of
$F\subset[0,1]$ is the variation of its indicator function,
\[
\chi(F)=\sup\sum_{i=0}^{n-1} |
\mathbh{1}_{t_{i}\in F}-\mathbh{1}_{t_{i+1}\in F}|,
\]
where the supremum is taken over partitions $0=t_0\leq
t_1\leq\cdots\leq t_n=1$. The quantity
\[
\mathsf{v}(F)=\sup_{\varphi\in\mathscr{C}^{1}, 0 \leq\varphi
\leq1} \int_{F}
\varphi'(x)\,dx,
\]
where $\mathscr{C}^{1}$ is the family of differentiable functions on $[0,1]$,
captures the number of components of $F$ with nonempty interiors, in
particular $\mathsf{v}(F)\leq\chi(F)$. Remark that $\mathsf{v}$ is
not a regularity
modulus, because a set $F$ with small $\mathsf{v}(F)$ can contain an
arbitrarily large number of isolated singletons.

%
\begin{theorem}
\label{thmset-bounded-variation}
If $p_{x,y}$ is a function of $x,y\in[0,1]$ such that
%
\begin{equation}
\label{eqperimeter} \sup_{\varphi\in\mathscr{C}^{1},0 \leq\varphi\leq1}
\int_{\mathbb{X}\times\mathbb{X}}p_{x,y}
\varphi'(x)\varphi'(y)\,dx\,dy=\infty,
\end{equation}
then there is no random closed set $\xi$ satisfying $\mathbf{E}
\chi(\xi)^2<\infty$ having $p_{x,y}$ as its two-point covering
function.
\end{theorem}

\begin{pf}
Let $\mathsf{H}$ be the family of functions
$h(x,y)=\varphi'(x)\varphi'(y)$ for
$\varphi\in\mathscr{C}^{1}$ with $0 \leq\varphi\leq1$. Then
\[
\mathsf{v}(F)^2=\sup_{h\in\mathsf{H}}\int
_{\mathbb{X}\times
\mathbb{X}}\mathbh{1}_{x,y\in F}h(x,y)\,dx\,dy =\sup
_{h\in\mathsf{H}} \mathsf{g}_h(F).
\]
Theorem~\ref{thrreal-reg} implies that $\Phi$ is realisable by a
random closed set $\xi$ with $\mathbf{E}\chi(\xi)^2<\infty$ if and
only if
\[
\sup_{h\in\mathscr{C}_{\mathrm{o}}} \Bigl[\inf_{F\in\mathcal
{X}}\bigl[
\chi(F)^2-\mathsf{g}_{h}(F)\bigr] +\Phi(
\mathsf{g}_{h}) \Bigr]<\infty.
\]
It implies in particular
\[
\sup_{h\in\mathsf{H}} \Bigl[\inf_{F\in\mathcal{X}}\bigl[\chi
(F)^2-\mathsf{g}_{h}(F)\bigr] +\Phi(\mathsf{g}_{h})
\Bigr]<\infty.
\]
Since $\chi(F)^2\geq\mathsf{v}(F)^2\geq\mathsf{g}_{h}(F)$ for
$h\in\mathsf{H}$,
this condition would imply that
\[
\sup_{h\in\mathsf{H}}\Phi(\mathsf{g}_{h})<\infty,
\]
contradicting (\ref{eqperimeter}). Thus $\Phi$ is not realisable.
\end{pf}

Further results on realisability of random sets can be found in
\cite{GalLac}, where it is shown that by relaxing the closedness
assumption it is possible to split the positivity and regularity
conditions as it was the case in Section~\ref{sechardcore}.

\section{Contact distribution functions for random sets}\label{seccont-distr-funct}

Results from Section~\ref{secreals-under-smoothn} concern
realisability of the two-point covering probabilities, which are
closely related to the values of the capacity functional (hitting
probabilities) on two-point sets. Here, we consider the realisability
problem for a capacity functional defined on the family of balls
in $\mathbb{R}^d$. If $T$ is the capacity functional of a random
closed set~$\xi$, then
\[
T\bigl(B_R(x)\bigr)=\mathbf{P}\bigl\{\xi\cap B_R(x)
\neq\varnothing\bigr\}
\]
is closely related to the spherical contact distribution function
$\mathbf{P}\{\mathbf{d}(x,\xi)\leq R|x\notin\xi\}$, $R\geq0$,
which is the
cumulative distribution function of the distance between $\xi$ and
$x$ given that $x\notin\xi$.

%
\begin{theorem}
\label{thrcdf}
A function $\tau_x(R)$, $R\geq0$, $x\in A\subset\mathbb{R}^d$, is
realisable as
$T(B_R(x))$ for a random closed set $\xi$ if and only if
%
\begin{equation}
\label{eqnBalls} \Phi(\mathsf{g})=\sum_{i=1}^m
a_i\tau_{x_i}(R_i)\geq0
\end{equation}
for all $m\geq1$, $x_1,\ldots,x_m\in A$ and $R_1,\ldots,R_m\geq0$,
such that the function
%
\begin{equation}
\label{eqsum-fneq-0} \mathsf{g}(F)=\sum_{i=1}^m
a_i\mathbh{1}_{B_{R_i}(x_i)\cap
F\neq\varnothing}\geq0,\qquad F\in\mathcal{F}
\end{equation}
is nonnegative.
\end{theorem}

\begin{pf}
The necessity is evident.

\textit{Sufficiency}. Let $\mathsf{G}$ be the vector space generated by
constants and functions $\mathsf{g}_{h,x}(F)=h(\mathbf{d}(x,F))$,
$F\in\mathcal{F}$,
where $\mathbf{d}(x,F)$ is the distance from $x\in\mathbb{R}^d$ to
the nearest
point of $F$, and $h$ is a continuous function on $\mathbb{R}$ with bounded
support. The functions $\mathsf{g}_{h,x}$ are all continuous in the Fell
topology, since the Fell topology in $\mathbb{R}^d$ coincides with the
topology of pointwise convergence of distance functions $\mathbf{d}(x,F)$
for $x\in\mathbb{R}^d$; see \cite{mo1}, Theorem~B.12.

It suffices to show that $\Phi$ is positive on $\mathsf{G}$. Let
$\mathsf{g}(F)=\sum_{i=1}^m a_i h_i(\mathbf{d}(x_i,F))$. Uniform
approximation of
$h_1,\ldots,h_m$ by step functions on their supports yields a
function $\hat{\mathsf{g}}$ of the form (\ref{eqsum-fneq-0}) so that
$\hat{\mathsf{g}}(F)\geq-\varepsilon$ for some $\varepsilon>0$.
Letting $\varepsilon\downarrow0$
and using (\ref{eqnBalls}) yield that
\[
\Phi(\mathsf{g})=\sum_{i=1}^m
a_i \int h_i(t)\,d\tau_{x_i}(t)\geq0.
\]\upqed
\end{pf}

If $\tau_x(R)=\tau(R)$ does not depend on $x$, it may be possible to
realise it as the contact distribution function of a stationary random
closed set. If the argument $x$ of $\tau_x(R)$ takes only a single
value, then the necessary and sufficient condition on $\tau_x(\cdot)$
is that it is a nondecreasing right-continuous function with values
in $[0,1]$, that is, the cumulative distribution function of a
sub-probability measure on $\mathbb{R}_+$. The following result
concerns the
case of $x$ taking two possible values.

%
\begin{theorem}
Let $x_1,x_2\in\mathbb{R}^d$, with $l=\|x_{1}-x_{2}\|$, and let
$\tau_{x_1}$ and $\tau_{x_2}$ be cumulative distribution
functions of two sub-probability measures on $\mathbb{R}_+$. Then there
exists a random closed set $\xi$ such that
$\tau_{x_i}(R)=T(B_R(x_i))$ for $r\geq0$ and $i=1,2$ if and only
if for all $r\geq0$
%
\begin{equation}
\label{eqmathc-taux2=-infep} \tau_{x_1}\bigl(\max(R-l,0)\bigr)\leq
\tau_{x_2}(R)\leq\tau_{x_1}(R+l).
\end{equation}
\end{theorem}

\begin{pf}
\textit{Necessity}. Let $\xi$ be a random closed set with
$\tau_{x_i}(R)=T(B_R(x_i))$. Let $a_1$ and $a_2$ be random points
such that $a_1,a_2\in\xi$ a.s. and $R_i=\mathbf{d}(x_i,a_i)=\mathbf{d}
(x_i,\xi)$,
$i=1,2$, have cumulative distribution functions $\tau_{x_1}$ and
$\tau_{x_2}$, respectively. Then $|R_1-R_2|\leq l$. Indeed, if, for
instance, $R_1>R_2+l$, then $a_2$ is nearer to $x_1$ than $a_1$
contrary to the assumption. Thus $R_1\leq R$ implies $R_2\leq R+l$,
so that $\tau_{x_1}(R)\leq\tau_{x_2}(R+l)$. The symmetry argument
with $x_1$ and $x_2$ interchanged yields~(\ref{eqmathc-taux2=-infep}).

\textit{Sufficiency}. Define two random variables $R_1$ and $R_2$
as inverse functions to $\tau_{x_1}$ and $\tau_{x_2}$ applied to a
single uniform random variable, so that
(\ref{eqmathc-taux2=-infep}) yields that $|R_1-R_2|\leq l$ a.s.
This means that none of the balls $B_{R_1}(x_1)$ and $B_{R_2}(x_2)$
lies in the interior of the other one. Now construct random closed
set $\xi$ consisting of two points: $a_1$ on the boundary of
$B_{R_1}(x_1)$ but outside of the interior of $B_{R_2}(x_2)$ and
$a_2$ on the boundary of $B_{R_2}(x_2)$ but outside of the interior
of $B_{R_2}(x_1)$. Then $a_1$ is nearest to $x_1$ and $a_2$ is
nearest to $x_2$ with given distributions of the distance.
\end{pf}

\begin{appendix}
\section*{Appendix: A combinatorial lemma}\label{secappend-comb-lemma}

\setcounter{equation}{0}
\setcounter{definition}{0}

Recall that $P_{t}(\mathbb{X})$ denotes the \emph{packing number} of
$\mathbb{X}$ with
metric $\mathbf{d}$, that is, the maximum number of points in the space
$\mathbb{X}$
with pairwise distance exceeding $t$; see~\cite{mati95}, page~78.

%
\begin{lemma}
\label{lemmaalgo}
If $Y=\sum\delta_{x_i}$ is a counting measure of total mass $n$,
then for all $t>0$,
\[
\sum_{i\neq j} \mathbh{1}_{\mathbf{d}(x_i,x_j)\leq t} \geq n
\biggl(\frac{n}{P_{t}(\mathbb{X})}-1 \biggr).
\]
\end{lemma}

\begin{pf}
Denote
\[
n(Y,x_i)=Y\bigl(B_t(x_i)\bigr)-1,
\]
where $B_t(x_i)$ is the closed ball of radius $t$ centred at $x_i$.
Furthermore, denote
\[
\mathsf{g}_{h_t}(Y)=\sum_{i\neq j}
\mathbh{1}_{\mathbf
{d}(x_i,x_j)\leq t}.
\]
Then
\begin{eqnarray*}
\mathsf{g}_{h_t}(Y-\delta_{x_i})&=&\mathsf{g}_{h_t}(Y)-2n(Y,x_i),
\\
\mathsf{g}_{h_t}(Y+\delta_{x_i})&=&\mathsf{g}_{h_t}(Y)+2n(Y,x_i)+2.
\end{eqnarray*}
Let $x_i$ and $x_j$ be two distinct points from the support of $Y$
with $\mathbf{d}(x_i,x_j)\leq t$. Assume that $n(Y,x_i)< n(Y,x_j)$ or
$n(Y,x_i)= n(Y,x_j)$ with $i<j$ and define
\[
Y'=Y-\delta_{x_j}+\delta_{x_i}
\]
obtained from $Y$ by transferring a mass $1$ from $x_j$ to $x_i$.
Call $Y''=Y-\delta_{x_j}$. Remark that $n(Y'',x_i)=n(Y,x_i)-1$
because $\mathbf{d}(x_i,x_j)\leq t$. Since $n(Y,x_j)\geq n(Y,x_i)$,
\begin{eqnarray*}
\mathsf{g}_{h_t}\bigl(Y'\bigr)&=&\mathsf{g}_{h_t}
\bigl(Y''\bigr)+ 2n\bigl(Y'',x_i
\bigr)+2
\\
&=&\mathsf{g}_{h_t}(Y)-2n(Y,x_j)+2n\bigl(Y'',x_i
\bigr)+2
\\
&=&\mathsf{g}_{h_t}(Y)-2n(Y,x_j)+2n(Y,x_i)-2+2
\\
&\leq&\mathsf{g}_{h_t}(Y).
\end{eqnarray*}
Furthermore, $n(Y',x_i)=n(Y,x_i)$ because the transferred mass
remains in the ball with centre $x_i$ and radius $t$, and
$n(Y',x_j)=n(Y,x_j)$ as well. Thus, $n(Y',x_i)\leq n(Y',x_j)$. Repeat the
mass transfer from $x_j$ to $x_i$ until the mass at $x_j$
disappears. Call the resulting counting measure $Y_1$.

Apply the same construction to $Y_1$ and repeat it until there are
no more distinct points at distance at most $t$. This happens in a
finite time because the cardinality of the support of $Y$ strictly
decreases at each step.

The obtained counting measure $\widehat{Y}$ is supported by a set of
points $\{y_{1},\ldots, y_{q}\}$ with pairwise distances
exceeding $t$. Thus,
\[
\mathsf{g}_{h_t}(Y)\geq\mathsf{g}_{h_t}(\widehat{Y})= \sum
_{i=1}^q m_i(m_i-1),
\]
where $m_i=\widehat{Y}(\{y_i\})$. Under the restriction $\sum_{i=1}^q
m_i=n$, the minimal value $\sum_im_i(m_i-1)$ is reached for
$m_i=n/q$, whence
\[
\mathsf{g}_{h_t}(Y)\geq n \biggl(\frac{n}{q}-1 \biggr).
\]
It remains to note that $q\leq P_{t}(\mathbb{X})$.
\end{pf}

It is also possible to define a counting measure by placing masses
from the interval $[n/q,n/q+1]$ at the points forming the packing net
of $\mathbb{X}$. Thus, there exists a counting measure $Y$ such that
\[
\mathsf{g}_{h_t}(Y)\leq n \biggl(\frac{n}{P_{t}(\mathbb{X})}+1 \biggr).
\]
\end{appendix}

\section*{Acknowledgements}\label{secacknowledgements}
The authors are grateful to John Quintanilla and Zbigniew Lipecki for
literature hints at early stages of this work and to Tobias Kuna for
comments on the preprint version. The comments of the referees and the
Editor greatly inspired the authors to improve the readability of the
paper.
Raphael Lachieze-Rey is grateful to the University of Bern for hospitality.


%

\printaddresses

\end{document}